\documentclass[11pt]{amsart}
\usepackage{amsfonts}
\usepackage{amssymb}
\usepackage{amsthm}
\usepackage{amsmath}
\usepackage{graphicx}
\usepackage{epsfig}
\usepackage{psfrag}
\usepackage{amssymb,latexsym}
\usepackage{amsmath,amscd}
\usepackage{fancyhdr}
\usepackage{indentfirst}
\usepackage{graphicx}
\usepackage{newlfont}
\usepackage{amsfonts}
\usepackage{amsthm}
\usepackage{mathrsfs}
\usepackage{psfrag}
\usepackage[all]{xy}
\usepackage[color=red!40]{todonotes}

\newcommand{\ft}{\mathfrak{t}}

\newcommand{\C}{\mathbb{C}}
\renewcommand{\P}{\mathbb{P}}
\newcommand{\CP}{\mathbb{CP}}
\newcommand{\R}{\mathbb{R}}
\newcommand{\Q}{\mathbb{Q}}

\newcommand{\Z}{\mathbb{Z}}
\newcommand{\M}{\mathcal{M}}
\newcommand{\pt}{\text{pt}}

\definecolor{red}{rgb}{.6,0,0}
\definecolor{green}{rgb}{0,.6,0}
\newboolean{draft}
\newcommand{\cmt}[1]
{\ifthenelse {\boolean{draft}}
{{\sc \tiny \color{red} #1}}
{}}

\newcommand{\margincmt}[1]
{\ifthenelse {\boolean{draft}}
{\marginpar{{\sc \tiny \color{red} #1}}}
{}}

\newtheorem{theorem}{Theorem}[section]
\newtheorem{lemma}[theorem]{Lemma}
\newtheorem{proposition}[theorem]{Proposition}

\newtheorem{remark}[theorem]{Remark}

\newtheoremstyle{example}
{9pt}{9pt}{}{}{\bfseries}{}{.5em}{}
\theoremstyle{example}

\begin{document}

\title[On the Gromov width of polygon spaces]{On the Gromov width of polygon spaces}

\author{Alessia Mandini}\address{Departamento de Matem\'atica, PUC-Rio
Rua Marqu\^es de S\~ao Vicente 225, G\'avea, CEP 22451-900 Rio de Janeiro, Brazil\\ alessia.mandini@mat.puc-rio.br}
\author{Milena Pabiniak}\address{Mathematisches Institut, Universit\"at zu K\"oln, Weyertal 86-90, D-50931 K\"oln, Germany
\\pabiniak@math.uni-koeln.de}

\begin{abstract} 
For generic $r=(r_1,\ldots,r_n) \in \R^n_+$ the space $\M(r)$ of $n$--gons in $\R^3$ with edges of lengths $r$ 
is a smooth, symplectic manifold. We investigate its Gromov width and prove that the expression 
$$2\pi \min \{2 r_j,   (\sum_{i \neq j} r_i) - r_j\,\,|\, j=1,\ldots,n\}$$
is the Gromov width of all (smooth) $5$--gon spaces and of $6$--gon spaces, under some condition on $r \in \R^6_+$.
The same formula constitutes  a lower bound for all (smooth) spaces of $6$--gons.
Moreover, we prove that the Gromov width of $\M(r)$ is given by the above expression when $\M(r)$ is symplectomorphic
to $\C\P^{n-3}$, for any $n
\geq 4$.
\end{abstract}

\maketitle

\tableofcontents

\section{Introduction}\label{section introduction}
In 1985 Mikhail Gromov proved his famous non-squeezing theorem saying that a ball
$B^{2N}(r)$ of radius $r$ in a symplectic vector space $\R^{2N}$ cannot be symplectically embedded into $B^2(R)\times \R^{2N-2}$
unless $r\leq R$ (both sets are equipped  with the usual symplectic structure induced from $\omega_{std} = \sum dx_j \wedge dy_j$ 
on $\R^{2N}$). This motivated the definition of the invariant called the Gromov width.
Consider the ball of capacity $a$
$$ B^{2N}_a = \big \{ z \in \C^N \ \Big | \ \pi \sum_{i=1}^N |z_i|^2 < a \big 
\} \subset \R^{2N}, $$
with the standard symplectic form
$\omega_{std} = \sum dx_j \wedge dy_j$.
The \textbf{Gromov width} of a $2N$-dimensional symplectic manifold $(M,\omega)$
is the supremum of the set of $a$'s such that $B^{2N}_a$ can be symplectically
embedded in $(M,\omega)$. It follows from Darboux's Theorem that the Gromov width 
is positive unless $M$ is a point.

Let $n$ be an integer greater or equal to $4$ and let $r_1, \ldots, r_n$ be positive real numbers. 
The {\bf polygon space} $\M(r)$ is 
the space of closed piecewise linear paths in $\R^3$ such that the $j$-th step
has norm $r_j$, modulo rigid motions. 
This moduli space 
is also the symplectic reduction of the product of $n$ spheres, of radii $r_1, \ldots, r_n$,
by the diagonal action of $SO(3)$,
$$
\M(r) = \mathcal{S}_r  / \!\! /_0 SO(3) =
\big\{ (\overrightarrow{e}_1, \dots, \overrightarrow{e}_n) \in \prod_{i=1}^n S_{r_i}^2 \mid \sum_{i=1}^n  \overrightarrow{e}_i=0 \big\}/SO(3).
$$
The length vector $r$ is {\bf generic} if and only if the scalar quantity 
$$ \epsilon_I(r):= \sum_{i \in I} r_i - \sum_{i \in I^c} r_i$$ 
is not zero for any  $I \subset \{ 1, \ldots, n \}$  ($I^c$ denotes $\{ 1, \ldots, n \} \setminus I$). In this case,
the polygon space $\M(r)$ is a smooth 
symplectic (in fact, K\"ahler) manifold of dimension $2(n-3)$.
Observe that for any permutation $\sigma \in S_n$ the manifolds $\M(r)$ and $\M(\sigma(r))$ are symplectomorphic.
Note that the existence of an index set $I$ such that $ \epsilon_I(r)=0$
is equivalent to the existence of an element $(\overrightarrow{e}_1, \dots, \overrightarrow{e}_n )$ in $\prod_{i=1}^n S_{r_i}^2 $
that lies completely on a line. The stabilizer of such an element is non-trivial:
it is the $S^1 \subset SO(3)$ of rotations around that line. Therefore the associated symplectic reduction 
has a singularity. 
An index set $I$ is called {\bf short} if $\epsilon_I(r)<0$, and {\bf long} if its complement is short. Moreover
$I$ is {\bf maximal short} if it is short and is not contained in any other short set.

From an algebro-geometric point of view, polygon spaces are identified with the GIT quotient
of $(\C\P^1)^n$ by $PSL(2, \C)$. This GIT quotient is a compactification of the configuration
space of $n$ points in $\C\P^1$ and, via the Gelfand--MacPherson correspondence, 
relates polygon spaces to the symplectic reductions of the Grassmannian of $2$-planes in $\C^n$
by the maximal torus $U(1)^n$ of the unitary group $U(n)$.

In this work we analyze the Gromov width of polygon spaces $\M(r)$ for $r \in \R_+^n$ generic, $n>3$, and prove the following results.
For simplicity of notation, for $r \in \R_+^n$, let 
$$\rho(r):= 2\pi \min \{2
r_j,\,\big(\sum_{i \neq j} r_i \big) - r_j\,|\, j=1,\ldots,n\}.$$

\begin{theorem}\label{theo:5gon}
 For any generic $r \in \R_+^5$, the Gromov width of the associated manifold of $5$-gons, $\M(r)$, is equal to $\rho(r)$.
% $$2\pi \min \{2
%r_j,\,\big(\sum_{i \neq j} r_i \big) - r_j\,|\, j=1,\ldots,5\}.$$
\end{theorem}

\begin{theorem}\label{theo:6gon}
For any generic $r \in \R_+^6$, the Gromov width of the associated manifold of $6$-gons, $\M(r)$, is at least $\rho(r)$.
%The Gromov width of $\M(r_1,\ldots,r_6)$ is at least 
%$$2\pi \min \{2r_j,\,\big(\sum_{i \neq j} r_i \big) - r_j\,|\, j=1,\ldots,6\}.$$
Moreover, if for a permutation $\sigma \in S_6$ such that $r_{\sigma(1)}\leq \ldots \leq r_{\sigma(6)}$, 
one of the following holds:
\begin{itemize}
 \item $\{1,2,3,4\}$ and $\{1,2,6\}$ are short for $\sigma(r)$, or
\item $\{1,2,6\}$ and $\{4,6 \}$ are long for $\sigma(r)$, or
\item $\{ 5,6 \}$ and $\{ 2,3,6\}$ are short for $\sigma(r)$
\end{itemize}
then the Gromov width of $\M(r)$ is equal to $\rho(r)$.
%$\M(r_1,\ldots,r_6)$ is equal to 
%$$2\pi \min \{2r_j,\,\big(\sum_{i \neq j} r_i \big) - r_j\,|\, j=1,\ldots,6\}.$$
\end{theorem}

\begin{theorem}\label{theo:projective}
Assume that there exists a maximal $r$-short index set $\{i_0\}$. 
In this case, $\M(r)$ is symplectomorphic to $(\C \P^{n-3}, 2 \big(\,(\,\sum_{i \neq i_0} r_i )  - r_{i_0}\big) \omega_{FS}),$ 
where $\omega_{FS}$ denotes the usual Fubini-Study symplectic structure and its Gromov width is 
$\rho(r)= 2 \pi \big(\,(\,\sum_{i \neq i_0} r_i )  - r_{i_0} \big)$.
%$$2\pi \min \{2r_j,\,\big(\sum_{i \neq j} r_i \big) - r_j\,|\, j=1,\ldots,n\} .$$
%In this case, as $\{i_0\}$ is maximal short, it holds that
%$\rho(r)= 2 \pi \big(\,(\,\sum_{i \neq i_0} r_i )  - r_{i_0} \big)$.
\end{theorem}

We conjecture that the Gromov width of polygon spaces, for any $n\geq 4$ and any $r$ generic
is $\rho(r)$.
%$$
%2\pi \min \{2r_j,\,\big(\sum_{i \neq j} r_i \big) - r_j\,|\, j=1,\ldots,n\}.
%$$
\begin{remark}
Note that as $\M(r)$ and $\M(\sigma(r))$ are symplectomorphic for all $\sigma \in S_n$, we can always assume that $r_1 \leq \ldots \leq r_n$.
With this assumption, for $n\geq 4$
\begin{align*}\rho(r)
%=2\pi \min \{2r_j,\,\big(\sum_{i \neq j} r_i \big) - r_j\,|\, j=1,\ldots,n\}
&=2\pi \min \{2r_1,\,\big(\sum_{i \neq n} r_i \big) - r_n\}\\
&=\begin{cases}
2 \pi \big((\sum_{i \neq n} r_i ) - r_n\big)& \textrm{ if $\{n\}$ is maximal short} \\
4\pi \,r_1& \textrm{ otherwise }.
\end{cases}
\end{align*}
%$$
%=\begin{cases}
%2 \pi \big((\sum_{i \neq n} r_i ) - r_n\big)& \textrm{ if $\{n\}$ is maximal short} \\
%4\pi \,r_1& \textrm{ otherwise }
%\end{cases}
%$$
\end{remark}

An important tool in the proof of the above results is a toric action, called the bending action, 
defined on a dense open subset of $\M(r)$ (possibly on the whole $\M(r)$).
Let $\overrightarrow{d}=\overrightarrow{e}_{i}+\ldots +\overrightarrow{e}_{i+l}$ be a choice of a diagonal of the polygons in $\M(r)$.
The circle action associated to $\overrightarrow{d}$ rotates the piecewise linear path 
$\overrightarrow{e}_{i}+\ldots +\overrightarrow{e}_{i+l}$ along the axis of the diagonal $\overrightarrow{d}$. 
This action is defined on the dense open subset of $\M(r)$ consisting of polygons $P$ for which the diagonal
$\overrightarrow{d}$ does not vanish.
In this way any system of $(n-3)$ non--intersecting diagonals gives a toric action of $(S^1)^{n-3}$ on a dense open subset of 
$\M(r)$ (where the respective diagonals do not vanish). For many $r$'s and for appropriate choices of diagonals 
the action can be defined on the whole $\M(r)$. Using the flow of this action one can construct symplectic 
embeddings of balls and thus obtain lower bounds for the Gromov width.
The bending action has a central role also in determining upper bounds for the Gromov width, as 
certain tools (for example, \cite{Lu}) are available for toric manifolds which are Fano, or blow ups of Fano toric 
manifolds at toric fixed points. Using a Moser-type continuity argument we obtain upper bounds for the Gromov width of some non-toric spaces $\M(r)$.
The upper bounds coincide with the lower bounds we have determined by embedding techniques, and so we determine
an explicit formula for the Gromov width of (some) polygon spaces.

One should mention here that there are efficient methods for finding the Gromov width (and for solving the more general 
problem of ball packings) of $4$-dimensional manifolds that do not require the use of toric geometry. 
In particular, the Gromov width of the spaces of $5$-gons could 
also be found using Propositions 1.9, 1.10 and Remark 1.11 of \cite{McDuffellipsoids} by McDuff. These methods, however, are specific for dimension $4$. 
Therefore, instead of using these tools, we use some tools from toric geometry as those can be applied in any dimension.

{\bf Organization.}
We start with describing the tools for finding the Gromov width in Section \ref{section gromov width}. In Section \ref{section polygon spaces} we carefully define the polygon spaces and various toric actions on them (or on their subsets). In Section \ref{section projective} we compute the Gromov with of these $\M(r)$ which are symplectomorphic to a projective space. Sections \ref{section 5gons} and \ref{section 6gons} are devoted to the computation of the Gromov width of spaces of $5$-gons and $6$-gons, respectively.

{\bf Acknowledgements:} The authors would like to thank Tara Holm, Yael Karshon, Dominic Joyce, Dusa McDuff, Felix Schlenk and Kazushi Ueda for helpful discussions.
The authors are very grateful to the referees for their corrections and useful comments.

The research leading to these results has received funding from the European Research Council 
under the European Union's Seventh Framework Programme (FP7/2007-2013) /
ERC Grant agreement no. 307119.
The second author was supported by the Funda\c{c}\~ao para a Ci\^encia e a Tecnologia (FCT, Portugal):
fellowship SFRH/BPD/87791/2012 and projects PTDC/MAT/117762/2010, EXCL/MAT-GEO/0222/2012.

\section{Gromov width}\label{section gromov width}
\subsection{Techniques for finding a lower bound for the Gromov width.}
We start with describing techniques for finding a lower bound for the Gromov width.
If a manifold $(M,\omega)$ is equipped with a Hamiltonian (so effective) action of a torus 
$T$, one can use this action to construct explicit embeddings of balls and 
therefore to obtain a lower bound for the Gromov width. 
Such a construction was provided by Karshon and Tolman in \cite{KarshonTolman}.
If additionally the action is {\bf toric}, that is, $\dim T= \frac 1 2 \dim M$, 
then more constructions are available
(see for example: \cite{Traynor}, \cite{Schlenk}, \cite{LMS}). 
In what follows we use 
results of Latschev, McDuff and Schlenk, \cite{LMS}, presented here as Proposition \ref{prop 
diamond} and Proposition \ref{prop diamondlike}.

Recall that a Hamiltonian action of a torus $T$ on a symplectic manifold $(M,\omega)$ 
gives rise to a moment map $\mu \colon M \rightarrow \mathfrak{t}^*$, to the dual of the Lie algebra of $T$, 
which is unique up to translation in $\mathfrak{t}^*$. 
Reparametrizing the torus $T$ by some automorphism $\psi$ of $T$ 
one obtains a Hamiltonian action of $T$ on $M$ with moment map $ \Psi \circ \mu$, 
where $\Psi \in GL(\dim T, \Z)$ denotes the map induced by the automorphism $\psi \colon T \rightarrow T$.
If $M$ is compact then the image, $\mu(M)$, is a Delzant polytope. 
Identifying $\mathfrak{t}^*$ with $\R^{\dim T}$ we can view $\mu(M)$ as a polytope in $\R^{\dim T}$. 
Note however that such an identification is not unique: it depends on the choice of splitting $T$ into a product of circles, 
and on the choice of identification of the Lie algebra of $S^1$ with the real line $\R$. 
Changing the splitting of $T$ results in applying a $GL(\dim T, \Z)$ transformation to $\R^{\dim T}$, 
while changing the identification $Lie(S^1) \cong \R$ results in rescaling. 
%Reparametrizing the torus $T$ by some automorphism of $T$ would also result in applying a $GL(\dim T, \Z)$ transformation.

As we are to calculate a numerical invariant, we need to fix a way of 
identifying the Lie algebra of $S^1$ with the real line $\R$. We think of 
the circle 
as $S^1 = \R / (2 \pi \Z)$. With this convention the moment map for the standard 
$S^1$-action on $\C$ by rotation with speed $1$ is given (up to an addition of a constant) by 
$z \mapsto -\frac 1 2 |z|^2$.  Define
$$ \Diamond^n(a):= \Big\{ (x_1,\ldots,x_n) \in \R^n ({\mathbf{x}}) \,|\, 
\sum_{j=1}^{n} |x_j| < \frac a 2 \Big\}\subset \R^n (\mathbf{x}).$$ When the 
dimension is understood from the context we simply write $\Diamond(a)$.
If $M^{2n}$ is toric, $\mu$ is the associated moment map and $\Diamond (a) 
\subset  \textrm{Int }\mu(M)$ is a subset of the interior of the moment map 
image, then a subset of $\mu^{-1}(\Diamond(a))=\Diamond(a) \times T^n$ is 
symplectomorphic to $\Diamond(a) \times (0,2\pi)^n \subset \R^n(\mathbf{x}) 
\times \R^n(\mathbf{y})$ with the symplectic structure induced from the 
standard 
one on $\R^n(\mathbf{x}) \times \R^n(\mathbf{y})$.
Below we present a result from Latschev, McDuff and Schlenk, (\cite[Lemma 4.1]{LMS}\footnote{This result was already used in \cite[Section 5]{Sch} and \cite[Section 9.3]{Schlenk} though not explicitly stated as a proposition.}) which, though stated in dimension $4$, holds 
also in higher dimensions. Note that the authors are using the convention where 
$S^1=\R/\Z$ and therefore the proposition below looks differently than \cite[Lemma 4.1]{LMS}. 
To translate the conventions observe that $\Diamond(a) \times (0, 2\pi)$ is symplectomorphic to $\Diamond(2\pi a) \times (0,1)$.

\begin{proposition}\cite[Lemma 4.1]{LMS} \label{prop diamond}
For each $\varepsilon >0$ the ball $B^{2n}_{2 \pi (a- \varepsilon)}$ of capacity $2 \pi (a- \varepsilon)$ symplectically 
embeds into $\Diamond^n(a) \times (0,2 \pi)^n \subset 
\R^n(\mathbf{x}) \times \R^n(\mathbf{y})$. 
Therefore, if for a toric manifold $(M^{2n},\omega)$ with moment map 
$\mu$,  
$$\Psi( \Diamond^n(a)) +x \subset \textrm{Int}\,\mu(M)$$ 
for some $ \Psi \in GL(n,\Z)$ and $x \in \R^n$,
 then the Gromov width of $(M^{2n}, \omega)$ is at least $2 \pi \,a$.
\end{proposition}
A more general result is true. 
Let $l_j<0<g_j$ be real numbers such that $g_j-l_j=a$, $j=1,\ldots,n$. We build 
a, not necessarily symmetric, cross whose arms are open intervals of length $a$ and take the convex hull 
of it. This way we obtain a ``diamond-like'' open subset 
$\underline{ \Diamond}^n(a)\subset \R^n (\mathbf{x})$. 
$$ \underline{ \Diamond}^n(a):= \underline{ \Diamond}^n(a)(l_1,g_1,\ldots,l_n,g_n)= Conv ( \cup_{j=1}^{n} \{ x_j \in 
(l_j,g_j), x_i=0 \textrm{ for }i\neq j\} ).$$
\begin{figure}[htbp]
\begin{center}
\psfrag{l1}{\footnotesize{$l_1$}}
\psfrag{l2}{\footnotesize{$l_2$}}
\psfrag{g1}{\footnotesize{$g_1$}}
\psfrag{g2}{\footnotesize{$g_2$}}
\includegraphics[width=4cm]{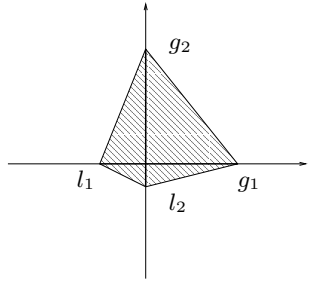}
\caption{Diamond-like shape $\underline{ \Diamond}^n(a)$,  $a=g_j-l_j$.}
\label{fig:diamond}
\end{center}
\end{figure}

\begin{proposition}\cite[Section 4.2]{LMS} \label{prop diamondlike} 
For each $\varepsilon >0$ the ball $B^{2n}_{2 \pi (a- \varepsilon)}$ of capacity $2 \pi (a- \varepsilon)$ symplectically 
embeds into $\underline{ \Diamond}^n(a)\times (0,2 \pi)^n \subset 
\R^n(\mathbf{x}) \times \R^n(\mathbf{y})$. Therefore, if  there exist $ \Psi \in GL(n,\Z)$ and $x \in \R^n$ such that 
$$\Psi( \underline{ \Diamond}^n(a)) +x \subset \textrm{Int}\,\mu(M)$$ 
for a toric manifold 
$(M^{2n},\omega)$ with moment map $\mu$, then the Gromov width of $(M^{2n}, 
\omega)$ is at least $2 \pi \,a$.\end{proposition}
Note that a simplex is a particular case of a diamond-like shape. Therefore the 
above Proposition also implies the following result 

\begin{proposition}\cite[Proposition 1.3]{Lu}\cite[Proposition 2.5]{Pabiniak}\cite[Lemma 5.3.1]{Schlenk}
Let $\Delta^n(a):=\{(x_1,\ldots,x_n) \in \R^n_{>0}\,|\,\sum_{k=1}^{n}x_k <a\}$ be the $n$-dimensional simplex. 
For any connected proper (not necessarily compact) Hamiltonian $T^n$ space $M$ let 
$$\mathcal{W}(\Phi(M)):= \sup \{a>0\,|\, \exists \Psi \in GL(n,\Z), x \in
\R^n\,s.t. \Psi (\Delta^n(a))+x \subset \Phi(M) \},$$
where $\Phi$ is some choice of moment map. Then the Gromov width of $M$ is at least $2 \pi \mathcal{W}(\Phi(M))$.
 \end{proposition}
 
Note that \cite{Pabiniak}, \cite{LMS} and \cite{Schlenk} use different identification of $\ft$ with $(\R^n)^*$ than \cite{Lu} and we here.

\subsection{Techniques for finding an upper bound for the Gromov width.}
It was already observed by Gromov that one can use J-holomorphic curves to find upper bounds for the Gromov width 
(see Proposition \ref{thm ub jholom}).
Many tools for finding upper bounds are based on a similar idea: non-vanishing of a certain Gromov-Witten type 
invariant implies some upper bound for the Gromov width. 
We start this section with explaining the above observation in more details. 
Later we recall some tools for finding upper bounds for the Gromov width constructed by Lu (\cite{Lu}), making use of a toric action.

\subsubsection{J-holomorphic curves and upper bounds of the Gromov width.}
Here we set the essential definitions and notations to be able to use pseudoholomorphic
curves and Gromov-Witten invariants to compute the upper bound of the Gromov width. 
We refer to McDuff--Salamon \cite{MSholom} for a comprehensive exposition of the subject, and to 
Caviedes \cite{Caviedes}, Zoghi \cite{Zoghi}, Karshon-Tolman \cite{KarshonTolman}
where these techniques are used to determine the Gromov width of certain coadjoint orbits.

An almost complex structure on a symplectic manifold $(M^{2n}, \omega)$ is a smooth fiberwise linear map
$J:TM \to TM$ such that $J^2=-Id$. An almost complex structure $J$ is $\omega$-compatible
if $g(v,w)= \omega(v,Jw)$ defines a Riemannian metric. Denote by $\mathcal{J}(M, \omega)$
the space of all $\omega$-compatible almost complex structures.

Let $(\C\P^1, j)$ be the Riemann sphere equipped with its standard complex structure $j$ and 
let $J$ be a $\omega$-compatible almost complex structures on $M$. A {\bf $J$-holomorphic curve}
is a map $u: \C\P^1 \to M $ satisfying $J \circ du = du \circ j$.
An important feature of $J$-holomorphic curves is that they come in families, which combine together 
in a moduli space as follows.
Given a homology class $A \in H_2(M; \mathbb Z)$, 
let $\M_A(M, J)$ denote
the moduli space 
of simple $J$-holomorphic curves:
$$\M_A(M, J)= \{ u: \C\P^1 \to M \mid u \text{ is a $J$-hol. curve}, u_*[\C\P^1]=A, u \text{ is simple} \}.$$

Consider the evaluation map
$$
\begin{array}{ccc}
\M_A(M, J) \times \C\P^1& \to& M\\
(u,z) & \to & u(z).
\end{array}
$$

The group $PSL(2,\C)$ acts naturally on $\C\P^1$ and by reparametrization on $\M_A(M, J) $. The
evaluation map descends to the quotient and we obtain the map
$$
ev_J:\M_A(M, J) \times_{PSL(2,\C)} \C\P^1 \to M.
$$

The following result explains how J-holomorphic curves 
can be used to obtain the upper bounds for the Gromov width. The idea goes back to Gromov and was used by him to prove 
his famous non-squeezing theorem. 
\begin{proposition}\cite[Proposition 3.6]{Zoghi}\label{thm ub jholom}
 Let $(M^{2n}, \omega)$ be a compact symplectic manifold. Given $A \in H_2(M; \mathbb Z)\setminus \{0\}$,
if for a dense open subset of $\omega$-compatible almost complex structures $J$ the evaluation map $ev_J$ is onto, 
then for any symplectic embedding $B^{2n}(a) \hookrightarrow M$ of a ball of radius $a$, (capacity $\pi a^2$), one has
$$ \pi a^2 \leq \omega(A)$$
where $\omega(A)$ is the symplectic area of $A$. In particular, it follows that the Gromov width of
$(M^{2n}, \omega)$ is at most $\omega(A)$.
\end{proposition}

One way to prove that the evaluation map is onto is via Gromov-Witten invariants.
Fix $A \in H_2(M; \mathbb Z)\setminus \{0\}$.
Let $\mathcal{J}_{reg}(M,\omega,A)$ denote the set of $\omega$-compatible almost complex structures
which are regular for $A$ in the sense of Definition 3.1.1 of \cite{MSholom}. For $J\in \mathcal{J}_{reg}(M,\omega,A)$ the set
$\M_A(M, J)$ is a smooth manifold. 
Let $\alpha=\alpha_1 \times \ldots\times \alpha_k$ be 
an element of $H_d(M^k; \Z)$ where $d$ is such that
\begin{equation}\label{GWittendim}
d+ ( \dim M + 2 c_1(TM)[A] +2k-6)= \dim M^k.
\end{equation}
The Gromov-Witten invariant%\footnote{The Gromov-Witten invariants can also be defined in a more general setting and then they are rational numbers.}
$$\Phi_A(\alpha_1,\ldots, \alpha_k) \in \Q$$
``counts" (in nice situations it is an element of $\Z$) the number of J-holomorphic curves $u$ in the homology class $A$ which meet each of the cycles $\alpha_1,\ldots, \alpha_k$.
The precise definition of the Gromov-Witten invariant 
involves some delicate and technical tools that go beyond what is needed for the purpose of this work, 
so we refer the reader to \cite{MSholom}. 
We want to stress that Gromov--Witten invariants are symplectic invariants and independent of the choice of 
an almost complex structure $J \in \mathcal{J}_{reg}(M, \omega,A)$.

Let $[\text{pt}]$ denote the Poincar\'e dual of the homology class of a point.
If $\Phi_A([\pt],\alpha_2,\ldots, \alpha_k) \neq 0$ for some classes $\alpha_2,\ldots, \alpha_k$ then the 
evaluation map is onto and we can apply the above theorem.

In Section \ref{section 6gons} we will apply this method to $6$-dimensional (so semipositive) 
symplectic manifolds 
$(M,\omega)$ with homology classes $A$ for which
 $c_1(TM)[A]=2$. Then one can take $k=1$ and consider $\Phi_A([\pt]) \in \Z$. When $\Phi_A([\pt])\neq 0$ 
the above theorem implies that the Gromov width of $(M,\omega)$ is not greater than $\omega(A)$.

\subsubsection{Upper bounds for toric manifolds.}\label{sec ub toric}
Let 
$$\Delta=\bigcap_{i=1}^d\{x\in\mathbb{R}^n |\langle
  x,u_i\rangle\geq\lambda_i\}$$
be a Delzant polytope with primitive inward normals to the facets $u_1,\ldots,u_d$ and let $X_{\Delta}$ be the smooth toric symplectic manifold
corresponding to it. Let $\Sigma=\Sigma_{\Delta}$ be the fan associated to $\Delta$, and let $G(\Sigma)=\{u_1,\ldots,u_d\}$ 
denote the generators of the $1$-dimensional cones of $\Sigma$. A well-known construction in algebraic geometry 
assigns to $\Sigma$ a toric variety $X_{\Sigma}$ (no symplectic structure yet). Here $X_{\Sigma}$ is compact and 
smooth because $\Sigma$ is smooth and its support is the whole $\R^n$. Moreover, our $X_{\Sigma}$ is projective and 
therefore there is a one-to-one correspondence between K\"ahler forms on $X_{\Sigma}$ and strictly convex support 
functions $\varphi$ on $\Sigma$. Recall that a piecewise linear function $\varphi$ on $\Sigma$ is called a strictly 
convex support function for $\Sigma$ if 
\begin{itemize}
 \item[(i)] it is upper convex, i.e., $\varphi (x)+\varphi(y) \geq \varphi(x+y)$ for all $x,y \in \R^n$, and 
\item[(ii)] the restrictions of it to any two different $n$-dimensional cones $\sigma_1$, $\sigma_2 \in \Sigma$ are two 
different linear functions,
\end{itemize} (see \cite[Section 2]{Lu}). 
Given a support function $\varphi$ on $\Sigma$ the symplectic toric manifold $(X_{\Sigma}, 2 \pi \varphi)$ has moment map image 
$\Delta_{\varphi}$ defined by inequalities $\langle x,m \rangle \geq -\varphi(m)$ for all $m \in \R^n$. Therefore, the symplectic toric manifold, $X_{\Delta}$, obtained from $\Delta$ via the Delzant construction is $(X_{\Sigma}, 2 \pi \varphi)$ where $\varphi(u_i)=-\lambda_i$. To the pair $(\Sigma, 2 \pi \varphi)$ Lu associates 
\begin{align*}
\Upsilon (\Sigma, 2 \pi \varphi):&=\inf \{\sum_{k=1}^d 2 \pi \varphi(u_k) a_k >0\,|\,\sum_{k=1}^d u_k a_k=0, a_k \in \Z_{\geq 0}, k=1,\ldots ,d \} \\
 &=\inf \{-\sum_{k=1}^d 2 \pi \lambda_k a_k >0\,|\,\sum_{k=1}^d u_k a_k=0, a_k \in
\Z_{\geq 0}, k=1,\ldots ,d \} 
\end{align*}
and use it to describe an upper bound for the Gromov width of toric Fano manifolds. 

A toric manifold is {\bf Fano} if the anticanonical divisor is ample.
We refer the reader to, for example, \cite{CLS} or \cite{K}, for more information about Fano varieties.
Here we only mention the properties that will be relevant to our results.
For compact symplectic toric manifolds one can determine whether it is Fano by looking at the moment map image.
As the property of being Fano is a property of the underlying toric variety, not of the symplectic structure, 
it is enough to analyze the fan associated to the moment map image.
A compact symplectic toric manifold $M^{2n}$, with associated fan $\Sigma$, is Fano 
if and only if there exists a monotone polytope
$$\Delta_{mon}=\{x \in \R^n\,|\,\langle x, u_j\rangle \geq -1, \,j=1,\ldots, d\},$$
(the vectors $u_1, \ldots, u_d$ are primitive inward normals to the facets of $\Delta_{mon}$), 
whose fan is also $\Sigma$. This follows from Theorem 8.3.4 of \cite{CLS}.
Another way to see that is by observing that the dual 
$\Delta_{mon}^*=\{y \in \R^n\,|\,\langle x, y \rangle \geq -1\}$
 is 
exactly equal to the convex hull
of the points $\{u_j, \,j=1,\ldots,d\}$ and applying Proposition 3.6.7 of \cite{K}.
In particular all monotone compact symplectic toric manifolds are Fano.

We now quote a result of Lu which we will use to find upper bounds of the Gromov width.
\begin{theorem}\cite[Theorem 1.2]{Lu}\label{theorem LuFanoUpperBounds}
 If $X_{\Delta}=(X_{\Sigma}, 2 \pi \varphi)$ is Fano then the Gromov width of $X_{\Delta}$ is at most 
$$\Upsilon (\Sigma, 2 \pi \varphi)=\inf \{-\sum_{k=1}^d  2 \pi \lambda_k a_k >0\,|\,\sum_{k=1}^d u_k a_k=0, a_k \in
\Z_{\geq 0}, k=1,\ldots ,d \} $$
\end{theorem}

Not all polygon spaces are toric and Fano. Some of the not Fano ones can be obtained from toric Fano 
manifold by a sequence of toric blow ups. In these situations we can apply another theorem of Lu.

\begin{theorem}\cite[Theorem 6.2]{Lu}\label{theorem LuBlowupUpperBounds}
Let $X_{\widetilde{\Sigma}}$ be a toric manifold obtained from a toric Fano manifold $X_{\Sigma}$ by a sequence of blow ups at toric 
fixed points. Then the generators of $1$-dim cones of associated fans satisfy $G(\Sigma) \subset G(\widetilde{\Sigma})$. Moreover any strictly convex support function $\varphi$ for $\Sigma$ is also strictly convex for $\widetilde{\Sigma}$ and it holds that the Gromov width of $(X_{\widetilde{\Sigma}}, 2 \pi \varphi)$ is not greater than $\Upsilon(\Sigma,2 \pi \varphi)$.
\end{theorem}
Note the typo in \cite{Lu}: there is an extra $2\pi$ appearing in his formulation of the above theorem.

\section{Polygon spaces}\label{section polygon spaces}
The moduli space  $\M(r)$, $r \in \R_+^n$, $n \geq 4$, of closed spatial polygons is the space of closed 
piecewise linear paths in $\R^3$ with the $j$-th step
of length $r_j$, modulo rigid motions in $\R^3$ (i.e. rotations and translations). 
The space $\M(r)$ inherits a symplectic structure by means of symplectic reduction, as we describe below.

For any choice of $n$ positive real numbers $r= (r_1, \ldots, r_n) \in \R^n_+,$ $n \geq 4$,
let $(S^2_{r_i},\omega_i) $ be the sphere in $\R^3$ of radius $r_i$ and center the origin,
equipped with the symplectic volume form. The product
$$\mathcal{S}_r = \Big(\prod_{i=1}^n S_{r_i}^2, \omega= \sum_{i=1}^n \frac{1}{r_i}
p_i^* \omega_i \Big),$$
where $p_j\colon \prod_{i=1}^n S_{r_i}^2
\rightarrow S_{r_j}^2 $ is the projection on the $j$-th factor, is a compact
smooth symplectic manifold.

The group $SO(3)$ acts diagonally on $\mathcal{S}_r$ via the coadjoint action
(thinking of each sphere $S_{r_i}^2$ as of a $SO(3)$-coadjoint orbit).
This action is Hamiltonian with moment map 
\begin{displaymath}
\begin{array}{rcl}
\mu : \mathcal{S}_r& \rightarrow &\mathfrak{so}(3)^* \simeq \R^3\\
(\overrightarrow{e}_1, \ldots, \overrightarrow{e}_n) & \mapsto & \overrightarrow{e}_1 + \cdots + \overrightarrow{e}_n.\\
\end{array}
\end{displaymath}
The symplectic quotient
$$\M(r):= \mathcal{S}_r  / \!\! /_0 SO(3) = \mu^{-1}(0) / SO(3)$$ 
is the space of $n$-gons of fixed side lengths  $r_1,\ldots, r_n$ modulo 
rigid motions, and is usually called polygon space.
When it generates no confusion we will use the name polygon for both: 
an element in $ \mu^{-1}(0)$ and its class in $\M(r)$.

Note that if $n=1$ then the closing condition cannot be satisfied, if 
$n=2$ then $\M(r)$ is either empty or a point,
depending on whether $r_1=r_2$ or not, and if $n=3$ then $\M(r)$ is either 
empty or a point, depending on whether $r_1,r_2,r_3$ satisfy a triangle inequality.
In our study of the Gromov width of polygon spaces we omit these trivial cases
and assume that $n \geq 4$.

A polygon is \emph{degenerate} if it lies completely on a line.
The moduli space $\M(r)$ is a smooth manifold if and only if the lengths vector 
$r$ is \emph{generic}, i.e. for each $I \subset \{ 1, \ldots, n \},$ the quantity 
$$ \epsilon_I(r):= \sum_{i \in I} r_i - \sum_{i \in I^c} r_i$$ is non-zero.
Equivalently, $r$ is generic if and only if in $\M(r)$ there are no degenerate polygons. In
fact, if there exists a polygon $P$ on a line (or an index set $I$ such that
$\epsilon_I(r)=0$) then its stabilizer is $S^1 \subset SO(3)$ since the polygon $P$ is fixed
by rotations around the axis it defines. Therefore the $SO(3)$-action on
$\mu^{-1}(0)$ is not free and the quotient, $\mu^{-1}(0)/SO(3)$, has singularities.
Note that, for $r$ generic, the polygon space $\M(r)$ inherits a symplectic form
by symplectic reduction.
Observe moreover that for any permutation $\sigma \in S_n$, the manifolds $\M(r)$ and $\M(\sigma(r))$ are symplectomorphic.

For any $r$ generic, an index set $I$ is said to be {\bf short} if $\epsilon_I(r)<0$, and {\bf long}
if $\epsilon_I(r)>0$, i.e. if its complement is short. 
An index set $I$ is {\bf maximal short} if it is short and maximal with respect
to the inclusion on the collection of short sets for $r$, i.e. any index set containing $I$ as a non-trivial subset is long.

In \cite{HausmannKnutson}, Hausmann and Knutson prove that polygon spaces
are also realized as symplectic quotients of the Grassmannians $Gr(2,n)$ of $2$-planes in $\C^n$, obtaining
the Gelfand--MacPherson's correspondence. The construction goes as follows. 
Let $U(1)^n$ be the maximal torus of diagonal
matrices in the unitary group $U(n)$ and consider 
the action by conjugation of $U(1)^n \times U(2) \subset U(n) \times U(2)$ on
$\C^{2n}$. As the diagonal circle $U(1) \subset U(1)^n \times U(2)$ acts trivially, let us consider
the effective action of  $K := \big(U(1)^n \times U(2)\big) / U(1)$ on $\C^{2n}$.
Let $q=(q_1, \ldots, q_n)$, with  $q_i= (c_i, d_i)^t \in \C^2$, denote the coordinates in $\C^{2n}$.
The Hamiltonian action of $K$ on $\C^{2n}$ 
$$q \cdot  [e^{i\theta_1}, \ldots, e^{i\theta_n}, A]= (A^{-1} q_1 e^{i\theta_1}, \ldots, A^{-1} q_n e^{i\theta_n}),$$
with $(e^{i\theta_1}, \ldots, e^{i\theta_n}, A)\in U(1)^n \times U(2)$,
has moment map
\begin{equation} \label{eq:mommap}
 \begin{array}{rcl}
\mu: \C^{2n} & \rightarrow &  \big( \mathfrak{u}(1)^n \big)^* \oplus  \mathfrak{su}(2)^*\\ 
q & \mapsto &  \Big(\frac{1}{2} |q_1|^2, \ldots, \frac{1}{2} |q_n|^2 \Big) \oplus  \displaystyle\sum_{i=1}^n (q_i q_i^*)_0,
\end{array}
\end{equation}   
 where $(q_i q_i^*)_0$ denotes the traceless part: $(q_i q_i^*)_0=q_i q_i^*- \textrm{Trace} (q_i q_i^*) \cdot Id.$

The polygon space $\M(r)$ is then symplectomorphic to the symplectic reduction 
of $\C^{2n}$ by $K$:
$$\M(r) = \C^{2n} \big/ \!\!\big/_{\!(r,0)} K$$
(cf \cite{HausmannKnutson}). 
In fact, performing the reduction in stages and 
taking first the quotient by $U(1)^n$ at the $r$-level set, one obtains 
the product of spheres $\mathcal{S}_r$. The residual $U(2)/U(1) \simeq SO(3)$ action is the
coadjoint action described above, and one recovers the 
polygon space $\M(r)$ as the symplectic quotient $\mathcal{S}_r/\!\!/_{\!0}
SO(3).$ 

Performing the reduction in stages in the opposite order leads to the
Gelfand--MacPherson correspondence. One first obtains the
Grassmannian $Gr(2,n)$ of complex planes in $\C^n$ as 
the reduction $ \C^{n \times 2} /\!\!/_{\!0} U(2)$. Then the quotient of $Gr(2,n)$ by the
residual $U(1)^n / U(1)$ action is isomorphic to the moduli space
of $n$ points in $\C\P^1$, cf. \cite{gm}, and is also isomorphic to the 
polygon space $\M(r)$, see \cite{klyachko, km}. This is summarized in the following diagram:
$$ \xymatrix{
 {}& \C^{n \times 2} \ar[dl]_{U(2)} \ar[dr]^{U(1)^n}& {}\\
 Gr(2,n) \ar[dr]_{U(1)^n / U(1)}  & {} &  \prod_{j=1}^n S_{r_j}^2
\ar[dl]^{U(2)/U(1) \simeq SO(3)} \\
{}& \M(r)   & {}
} $$

The chambers of regular values in the moment polytope $\Xi:= \mu_{U(1)^n} (Gr(2,n))$
are separated by walls $W_I=\{r\,|\,\epsilon_I (r)=0\}$ of critical values.
Note that an index set $I$ and its complement $I^c$ determine the same wall, and if $I$ 
has cardinality $1$ or $n-1$, then the associated wall $W_I$ is an external wall. 
A chamber $\mathcal C$ is called external if it contains in its closure an external wall. 
In particular, if $r$ is in an external chamber, then there is 
a maximal short index set $I$ of cardinality one. In this case the polygon space
$\M(r)$ is diffeomorphic to the projective space $\C\P^{n-3}$, \cite{Mandini}.

It is interesting to notice (though we will not use that in this work) that two polygon spaces are diffeomorphic if and only if their integral cohomology rings are isomorphic. This follows from \cite[Theorem 3]{FHS}.

\subsection{Bending action}\label{section bending}
Let $r\in \R^n_{+}$ be generic. For any polygon $P$ in $\M(r)$
of edges $\overrightarrow{e}_1, \ldots, \overrightarrow{e}_n$ and vertices $v_1, \ldots, v_n$, consider 
the system of $n-3$ non-intersecting 
diagonals $\overrightarrow{d}_1, \ldots, \overrightarrow{d}_{n-3}$ from the first vertex to the remaining
non-adjacent vertices, i.e. $\overrightarrow{d}_i (P)= \overrightarrow{e}_1 + \cdots + \overrightarrow{e}_{i+1}$. 
Following Nohara--Ueda \cite{nu}, we call this system of diagonals {\bf caterpillar system}.
The lengths of the $n-3$ diagonals
\begin{equation}\label{eq:lengths}
  \begin{array}{rcl}
(d_1, \ldots, d_{n-3}): \M(r) &\rightarrow & \R^{n-3}\\
P & \mapsto & (|\overrightarrow{d}_1 (P)|, \ldots,  |\overrightarrow{d}_{n-3} (P)|)
 \end{array}
\end{equation}
are
continuous functions on $\M(r)$ and are smooth on the subset where they are not zero. 
Their image is a convex polytope in $\R^{n-3}$, which we denote by $\Delta$, consisting of points 
$(d_1, \ldots, d_{n-3}) \in \R^{n-3}$
that satisfy the following triangle inequalities
\begin{equation}\label{eq:triangle ineq.}
\begin{aligned}
r_{i+2} &\leq d_i+ d_{i+1}\\
d_i &\leq r_{i+2}+ d_{i+1}\\
d_{i+1} &\leq r_{i+2} + d_i
\end{aligned}
\end{equation}
where $i=0,\ldots, n-3$ and we use the notation $d_0=r_1$, $d_{n-2}=r_n$.
The functions $d_i$ give rise to
Hamiltonian flows, called bending flows, (cf \cite{klyachko,km}). The circle action
associated with a given diagonal $\overrightarrow{d}_i$ is defined on the dense open subset $\{d_i \neq 0\} \subset \M(r)$ in the following way. The first $i+1$ edges
bend along the diagonal $\overrightarrow{d}_i$ at a constant speed while the remaining
edges do not move. 
Putting together the actions coming from $(n-3)$ non-intersecting diagonals we obtain a toric action of $(S^1)^{n-3}$ on the dense open subset 
$\{d_i \neq 0, i=1,\ldots, n-3\} \subset \M(r)$.
The action-angle coordinates
are given by the lengths $d_i, \ldots, d_{n-3}$ and the angles of rotation respectively.
If each $d_i$ does not vanish on $\M(r)$ then $\M(r)$ is a symplectic toric manifold.
In this case, the moment map is given by \eqref{eq:lengths}
and the moment polytope is $\Delta$ described by inequalities \eqref{eq:triangle ineq.}.

The choice of another system of $n-3$ non-intersecting diagonals gives rise 
to a different Hamiltonian system on an open
dense subset of $\M(r)$ (possibly the whole $\M(r)$). These actions were investigated in \cite{nu}, where it is shown that 
any bending action on polygon spaces is induced by an integrable system of Gelfand-Cetlin type
on the Grassmannian $Gr(2,n)$.

\section{Projective spaces}\label{section projective}

In this section we analyze the Gromov width of $\M(r)$ in the cases when $\M(r)$ is diffeomorphic to $\C \P ^{n-3}$, $n \geq 4$, 
and prove Theorem \ref{theo:projective}.

We assume that $r \in \R_+^n$ is generic (so $\M(r)$ is a smooth manifold) and
that $r_1 \leq r_2 \leq \ldots \leq r_n$.
As shown in \cite[Proposition 4.2]{Mandini}, $\M(r)$ is diffeomorphic to $\C \P ^{n-3}$
if and only if there is a maximal short set of cardinality 1. Using the assumption that the $r_i$'s are ordered non-decreasingly,
this is equivalent to 
 $\{1, n\}$ being long, i.e., $2r_1 > \gamma:= ( \sum_{j=1}^{n-1} r_j ) - r_n$.

\begin{proposition}\label{volume}
 Let $r$ be generic, ordered non-decreasingly and such that $\{1,n\}$ is long.
Then the symplectic volume of $\M(r)$ is 
$$
 \frac{(2 \pi)^{n-3}}{(n-3)!}  \gamma^{n-3}
$$
\end{proposition}
This proposition, together with  \cite[Proposition 4.2]{Mandini} recalled above, proves Theorem \ref{theo:projective}:
as the volume of the ball of dimension $2n-6$ and radius $\sqrt{2 \gamma}$ is equal to 
$\frac{\pi^{n-3}}{(n-3)!}(\sqrt{2 \gamma})^{2n-6}$,
Proposition \ref{volume} shows that the Gromov width of the above $\M(r)$ is $2 \pi \gamma$, which in this case is exactly $\rho(r)$.
%$$2\pi \min \{2r_j,\,\big(\sum_{i \neq j} r_i \big) - r_j\,|\, j=1,\ldots,n\} .$$
\begin{proof}
From \cite[Section 2.5.1, page 210]{Mandini}, we know that the symplectic volume of the polygon space
$\M(r)$ is given by
$$
vol \, \M(r) = C \sum_{(k_1, \ldots, k_n) \in K} \binom{n-3}{k_1, \ldots, k_n} r_1^{k_1}
 \cdots r_n^{k_n} \sum_{I \text{ long}} (-1)^{n-|I|} (\lambda_I^1)^{k_1} \cdots
 (\lambda_I^n)^{k_n}
$$
where $C= - \frac{(2 \pi)^{n-3}}{2(n-3)!}$, $K = \{( k_1, \ldots, k_n) \in
\mathbb Z^n_{\geq 0} \mid \sum_{i=1}^n k_i = n-3 \}$ and $\lambda_I^i=1$ if $i \in I$
and $\lambda_I^i=-1$ if $i \notin I$.
%
%For the long set $I=\{1, \ldots, n-1\}$ one gets a contribution to the volume of 
%\begin{align*}
% &-C \sum_{(k_1, \ldots, k_n) \in
%K} \binom{n-3}{k_1, \ldots, k_n} r_1^{k_1}\cdots (-r_n)^{k_n}\\=&-C(r_1+\ldots+r_{n-1}-r_n)^{n-3}=-C\gamma^{n-3}=\frac{(2 \pi)^{n-3}}{2(n-3)!}\gamma^{n-3}.
%\end{align*}
%
%

Let us analyze the contributions to the
coefficient of a generic element 
$r_{i_1}^{k_{i_1}} \cdots r_{i_l}^{k_{i_l}}$,
for some $l=1, \ldots, n-3$.
Note first that the contribution coming from the long set $I=\{1, \ldots, n-1\}$ is
\begin{equation}\label{first contribution}
C (-1)\cdot (-1)^{k_{n}}.
\end{equation}

Any long set $I$ different from $\{1, \ldots, n-1\}$ contains $n$. 
The index sets $I$ that contain $n$ and $ i_1, \ldots, i_l$ contribute by
\begin{equation*}
  C\! \!\sum_{\stackrel{I \text{ long}}{\{i_1, \ldots, \, i_l,n\} \subset I} } (-1)^{n-|I|} (\lambda_I^1)^{k_1} \cdots
 (\lambda_I^n)^{k_n} =C \! \!\sum_{\stackrel{I \text{ long}}{\{i_1, \ldots, \, i_l,n\} \subset I} } (-1)^{n-|I|} $$%=\sum_{j=0}^{n-1-l} (-1)^{n-j-l-1} \binom{n-1-l}{j}$$
 $$=\left\{ \begin{array}{ll}
 C\sum_{j=0}^{n-1-l} (-1)^{n-j-l-1} \binom{n-1-l}{j} & \textrm{ if }n\notin \{i_1,\ldots,i_l\} \\
 & \\
C \sum_{j=0}^{n-l} (-1)^{n-j-l} \binom{n-l}{j} & \textrm{ if }n\in \{i_1,\ldots,i_l\}.
 \end{array} \right.
\end{equation*}
The right hand side can be rewritten as
$$ C (-1)^{n-l-1} \sum_{j=0}^{n-1-l} (-1)^{j} \binom{n-1-l}{j}= C (-1)^{n-l-1} (1-1)^{n-1-l}=0,$$
in the first case, and similarly as $$C (-1)^{n-l} (1-1)^{n-l}=0$$ in the second case.

Long sets $I$ that contain $n$ and $l-1$ elements in $\{ i_1, \ldots, i_l\}$ contribute, up to a sign, by
\begin{equation*}
 \left\{ \begin{array}{ll} 
          C\sum_{j=0}^{n-1-l} (-1)^{n-j-l} \binom{n-1-l}{j} = C(-1)^{n-l} (1-1)^{n-1-l}=0 &  \text{if } n \notin \{i_1,\ldots,i_l\} \\
          & \\
         C\sum_{j=0}^{n-l} (-1)^{n-j-l+1} \binom{n-l}{j}=  C(-1)^{n-l+1} (1-1)^{n-l}=0 &\text{if } n\in \{i_1,\ldots,i_l\}.
         \end{array} \right.
\end{equation*}

Continuing this way we reach the long sets $I$ that contain $n$ and \emph{exactly one} element of $\{ i_1, \ldots, i_l\}$. 
If $n \notin \{i_1,\ldots,i_l\} $, then any set $I$ containing $n$ and one element of $\{ i_1, \ldots, i_l\}$ is necessarily long, 
and the contribution of such sets is still $0$.
If $n\in \{i_1,\ldots,i_l\}$, 
then the set $I$ containing $n$ and exactly one element of $\{ i_1, \ldots, i_l\}$ is long if and only if $\{i_1, \ldots, \, i_l\} \cap I=\{n\}$.
The contribution of these long sets $I$ is
$$
  C\!\!\!\!\!\!\!\!\sum_{\stackrel{I \text{ long}}{\{i_1, \ldots, \, i_l\} \cap I=\{n\}} } (-1)^{n-|I|}(-1)^{k_{i_1}+ \ldots+ k_{i_l}-k_n}=
  C\!\sum_{j=1}^{n-l}(-1)^{n-3-k_n} (-1)^{n-j-1}\binom{n-l}{j}$$
  $$=C(-1)^{k_n}((1-1)^{n-l}-1)=-C(-1)^{k_n}.
$$

We are left with analyzing the contributions from the sets $I$ that contain $n$ and no element of $\{ i_1, \ldots, i_l\}$.
This is possible only if $n \notin \{ i_1, \ldots, i_l\}$, in which case $k_n=0$. The contribution of such sets is
$$C\!\!\!\!\!\!\!\!\sum_{\stackrel{I \text{ long}}{i_1, \ldots, \, i_l \notin I,\ n\in I} } 
(-1)^{n-|I|} (-1)^{k_{i_1}+ \ldots+ k_{i_l}}
=C(-1)^{n-3}\sum_{j=1}^{n-1-l} (-1)^{n-j-1} \binom{n-1-l}{j}$$
$$=C(-1)^{2n-4}(0-1)=-C=-C(-1)^{k_n}.$$
Note that now the summation is starting at $j=1$ as $\{n\}$ is not long.

This way we proved that the contribution to the
coefficient of a generic element 
$r_{i_1}^{k_{i_1}} \cdots r_{i_l}^{k_{i_l}}$, given by longs sets $I$ that are different from $\{1, \ldots, n-1\}$ is 
$$-(-1)^{k_n}\,C ,$$
regardless of whether $n\in \{ i_1, \ldots, i_l\}$ or not.
Together with equation \eqref{first contribution} this implies that
\begin{align*}
vol \, \M(r)&=-2C \sum_{(k_1, \ldots, k_n) \in
K} \binom{n-3}{k_1, \ldots, k_n} r_1^{k_1}\cdots (-r_n)^{k_n}
\\=&-2C(r_1+\ldots+r_{n-1}-r_n)^{n-3}=-2C\gamma^{n-3}=\frac{(2 \pi)^{n-3}}{(n-3)!}\gamma^{n-3}.
\end{align*}
\end{proof}

\begin{remark}
Here is an alternative way of finding the Gromov width in this case. One can show that if $r$ is generic, 
ordered non-decreasingly and $\{1,n\}$ is long, then
 the bending action coming from the caterpillar system of diagonals is toric on $\M(r)$. The moment map image is 
the set $\Delta$ determined by the following inequalities from \eqref{eq:triangle ineq.}:
 $$r_2-r_1 \leq d_1, \ d_{n-3} \leq r_n +r_{n-1},\ |r_{k+1}-d_{k-1}| \leq d_k,\ k=2,\ldots, n-3.$$
After appropriate translation, this set is $GL(n,\Z)$ -equivalent to a simplex $\Delta^{n-3}(\gamma)$, namely
 $F(\Delta^{n-3}(\gamma))=\Delta$ where $F \colon \R^{n-3} \rightarrow  \R^{n-3}$
$$F(x)=
\left[\begin{array}{ccccc}
-1 & & &  &0\\
-1&-1& &  &0\\
& & \ddots &&0\\
-1&-1&\ldots &-1&0\\
0&0& \ldots &0&1
\end{array}\right]x
+ \left( \begin{array}{c}
r_1+r_2 \\r_1+r_2+r_3 \\ \vdots \\ r_1+\ldots +r_{n-3}\\ r_n-r_{n-1}
\end{array} \right).
$$
This proves that in this case the manifold $\M(r)$ is symplectomorphic to $(\C \P^{n-3}, 2  \,(\,( \sum_{j=1}^{n-1}r_j) -r_n) \omega_{FS})$
and its Gromov width is $2 \pi  \,(\,(\sum_{j=1}^{n-1}r_j)-~r_n)$.
  \end{remark}

\subsection{Gromov width of $4$-gons.}
Let $r \in \R^4_+$ be generic and without loss of generality assume that the lengths are non-decreasingly ordered. 
On $\M(r)$ consider the bending action along the diago\-nal $\overrightarrow{d}= \overrightarrow{e_1} + \overrightarrow{e_2}$.
The diagonal $\overrightarrow{d}$ does not vanish if 
$r_1 \neq r_2$ or $r_3 \neq r_4$, which is always the case by the genericity assumption.
Thus the bending action is defined on the whole 
$\M(r_1,\ldots,r_4)$ making it a toric symplectic $2$-dimensional manifold. In particular Hausmann and Knutson in \cite{HausmannKnutson}
show that they are diffeomorphic to $\C\P^1$ and that
the moment map image is then the interval 
$$[\max(r_2-r_1,r_4-r_3), \min(r_1+r_2,r_3+r_4)] = [\max(r_2-r_1,r_4-r_3),r_1+r_2]$$ of 
length $ \min(2r_1,r_1+r_2+r_3-r_4)$.
Therefore the Gromov width of $\M(r)$ is $2 \pi \min(2r_1,r_1+r_2+r_3-r_4)$, as claimed in Theorem \ref{theo:projective}.

%%%%%%%%%%%%%%%%%%%%%%%%%%%%%%%%%%%%%%%%%%%%%%%%%%%%%%%%%%%%%%%%%%%%%%%%
\section{Gromov width of the spaces of $5$-gons}\label{section 5gons}

In this section we analyze the Gromov width of $\M(r)$ for generic $r\in \R_+^5$.
For this purpose we use the bending action along the caterpillar system of diagonals 
starting from the first vertex, as in Figure \ref{fig:diag5}.

\begin{figure}[htbp]
\begin{center}
\psfrag{1}{\footnotesize{$e_1$}}
\psfrag{2}{\footnotesize{$e_2$}}
\psfrag{3}{\footnotesize{$e_3$}}
\psfrag{4}{\footnotesize{$e_4$}}
\psfrag{5}{\footnotesize{$e_5$}}
\psfrag{d1}{\footnotesize{$d_1$}}
\psfrag{d2}{\footnotesize{$d_2$}}
\includegraphics[width=5cm]{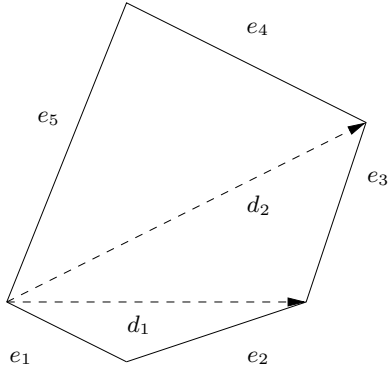}
\caption{Diagonals from the first vertex.}
\label{fig:diag5}
\end{center}
\end{figure}

Note that the caterpillar bending action on $\M(r)$ is toric if and 
only if $r_1 \neq r_2$ and $r_4\neq
r_5$. Since $\M(r)$ is symplectomorphic to $\M(\sigma(r))$ for any
permutation $\sigma \in S_5$ of the lengths vector, 
one can use this symplectomorphism to define a toric action on any $\M(r)$ with 
$r_{\sigma(1)} \neq r_{\sigma(2)} $ and $r_{\sigma(4)}\neq r_{\sigma(5)} $
 for some $\sigma \in S_5$.

The image of the bending flow functions \eqref{eq:lengths} (which is the moment map in the
toric case) is the polytope $\Delta $ in $\R^2$ given by 
the intersection of the rectangle of vertices 
$$ \begin{array}{cc}
 A=(|r_2 - r_1|, |r_5 -r_4|),&
B=(r_2 + r_1, |r_5 -r_4|),\\
C=(r_2 + r_1, r_5 + r_4),&
D=(|r_2 - r_1|, r_5+r_4)
\end{array} $$
with the non-compact region 
$$\Omega=\{(d_1,d_2) \in \R^2 \mid d_2 \geq d_1 - r_3,\, d_2 \geq -d_1 +r_3,\, d_2\leq d_1 
+r_3\}$$
as in
Figure \ref{fig:catMomentpolytope}, cf \cite{HausmannKnutson}.
\begin{figure}[htbp]
\begin{center}
\psfrag{A}{\footnotesize{$A$}}
\psfrag{B}{\footnotesize{$B$}}
\psfrag{C}{\footnotesize{$C$}}
\psfrag{D}{\footnotesize{$D$}}
\psfrag{omega}{\footnotesize{$\Omega$}}
\psfrag{x}{\footnotesize{$d_1$}}
\psfrag{y}{\footnotesize{$d_2$}}
\psfrag{r_3}{\footnotesize{$r_3$}}
 \psfrag{r_2+r_1}{\footnotesize{$r_1+r_2$}}
 \psfrag{r_2-r_1}{\footnotesize{$|r_1-r_2|$}}
 \psfrag{r_5+r_4}{\footnotesize{$r_4+r_5$}}
 \psfrag{r_5-r_4}{\footnotesize{$|r_5-r_4|$}}
\psfrag{y=x-r3}{\footnotesize{$d_2=d_1-r_3$}}
\psfrag{y=x+r3}{\footnotesize{$d_2=d_1+r_3$}}
\psfrag{y=-x+r3}{\footnotesize{$d_2=-d_1+r_3$}}
\includegraphics[width=10cm]{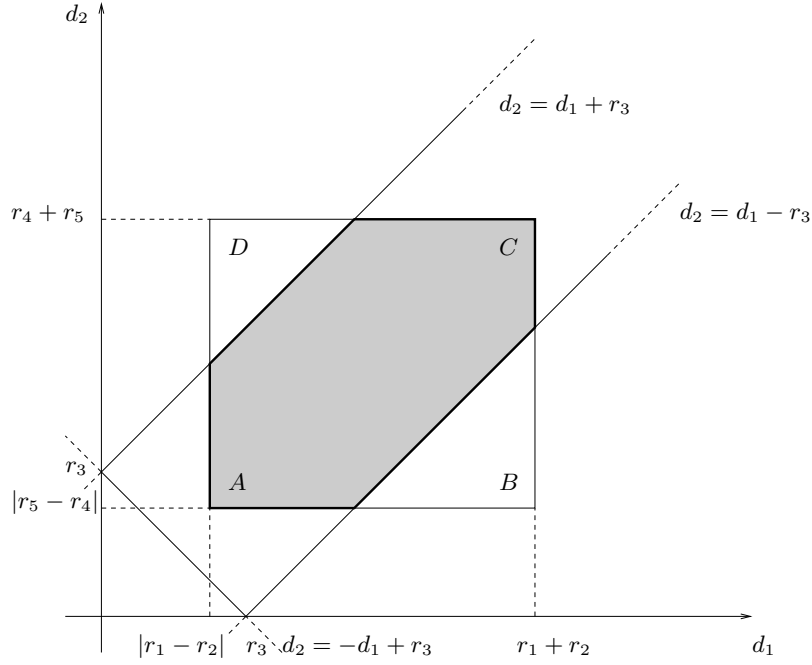}
\caption{Moment polytope for the caterpillar bending action on $\M(r)$.}
\label{fig:catMomentpolytope}
\end{center}
\end{figure}
The possible normals to the facets are
\begin{equation}\label{possible normals}
\begin{array}{llll}
u_1=(0,1), &u_2=(-1,1), &u_3=(-1,0),  &u_4=(0,-1), \\ 
u_5=(1,-1), &u_6=(1,0),  &u_7=(1,1). &
\end{array}
\end{equation}

Note that the diffeotype of $\M(r)$ is uniquely determined by the chamber $\mathcal C_r
\subset \Xi$ or, in other words, by the collection of $r$-short sets, cf. \cite{Mandini}.
Moreover, for $r$ in any fixed chamber $\mathcal{C}$, all toric $\M(r)$ have the same ``shape" of the moment map image with respect 
to the bending action along a fixed system of diagonals, i.e.  the moment polytopes have the same collection of facet normals, though the lattice lengths of the edges of the polytopes may vary.
Note that for a non-trivial reshuffling $\sigma(r)$, $\sigma \in S_5$, of the length vector $r$, $\sigma(r)$ is in a different chamber than $r$.
Nevertheless $\M(r)$ and $\M(\sigma(r))$ are symplectomorphic. The bending action along the caterpillar system of diagonals on $\M(\sigma(r))$ induces a Hamiltonian system on $\M(r)$ which may, or may not, correspond to bending along a different system of diagonals. In Section \ref{section 5gons upper}, when needed, we use reshuffling of the length vector $r$ to obtain an action defined on the whole $\M(r)$, making it a  toric manifold.
To determine the moment image, $\Delta$, the following table will be useful.
The table describes when the vertices of the rectangle $ABCD$ satisfy the inequalities defining the region $\Omega$ in the language of short sets. 
For simplicity we assume partial ordering on the length vector $r$: $r_1\leq r_2$, $r_4 \leq r_5$. All our reshuffled length vectors from Section \ref{section 5gons upper} will satisfy this partial ordering assumption.\\

\begin{tabular}{c|c|c|c}\label{5gon chart}
vertex &  $\in \{d_2 \geq d_1-r_3\} $ if & $\in \{d_2 \geq -d_1+ r_3\}$ if &  $\in \{d_2 \leq d_1+r_3\}$ if \\ \hline
$A$ &\{2,4\} is short & \{1,3,4\} is short &\{1,5\} is short\\ \hline
$B $&\{1,2,4\} is short & \{3,4\} is short &\{5\} is short\\ \hline
$C $&\{1,2\} is short & \{3\} is short & \{4,5\} is short \\ \hline
$D$ &\{2\} is short & \{1,3\} is short &\{1,4,5\} is short 
\end{tabular}\\

It is easy to see that under the assumption $0<r_1 \leq \ldots \leq r_5$, 
we can restrict our attention to the following 6 chambers:
\begin{itemize}
 \item $\mathcal C_1$, determined by the short sets:
$$\begin{array}{l}
\{i\} \quad \forall i=1, \ldots, 5,\\
 \{1,2\}, \{1,3\},\{1,4\},\{2,3\},\{2,4\},\{3,4\}, \\
\{1,2,3\},\{1,2,4\},\{1,3,4\},\{2,3,4\}.
\end{array}$$
Note that $\{5\}$ is maximal short.
For $r \in \mathcal C_1$, $\M(r)$ is diffeomorphic to $\C\P^2$.

\item $\mathcal C_2$, determined by the short sets:
$$\begin{array}{l}
\{i\} \quad \forall i=1, \ldots, 5,\\
\{1,2\}, \{1,3\},\{1,4\},\{1,5\},\{2,3\},\{2,4\},\{3,4\},\\
\{1,2,3\},\{1,2,4\},\{1,3,4\}.
\end{array}$$
For $r \in \mathcal C_2$, $\M(r)$ is diffeomorphic to $\C\P^2$ blown up at one point.

\item $ \mathcal C_3$, determined by the short sets:
$$\begin{array}{l}
\{i\} \quad \forall i=1, \ldots, 5,\\
\{1,2\}, \{1,3\},\{1,4\},\{1,5\},\{2,3\},\{2,4\},\{2,5\},\\
\{1,2,3\},\{1,2,4\},\{1,2,5\}.\end{array}$$
For $r \in \mathcal C_3$, $\M(r)$ is diffeomorphic to $\C\P^1 \times \C\P^1$.

\item $\mathcal C_4$, determined by the short sets:
$$\begin{array}{l}
\{i\} \quad \forall i=1, \ldots, 5,\\
\{1,2\}, \{1,3\},\{1,4\},\{1,5\},\{2,3\},\{2,4\},\{2,5\},\{3,4\},\\
\{1,2,3\},\{1,2,4\}.
\end{array}$$
For $r \in \mathcal C_4$, $\M(r)$ is diffeomorphic to $\C\P^2$ blown up at two points.

\item $\mathcal C_5$, determined by the short sets:
$$\begin{array}{l}
\{i\} \quad \forall i=1, \ldots, 5,\\
\{1,2\}, \{1,3\},\{1,4\},\{1,5\},\{2,3\},\{2,4\},\{2,5\},\{3,4\},\{3,5\},\\
\{1,2,3\}.
\end{array}$$
For $r \in \mathcal C_5$, $\M(r)$ is diffeomorphic to $\C\P^2$ blown up at three points.

\item $\mathcal C_6$, determined by the short sets: all $I$ with $|I|=1,2$.
For $r \in \mathcal C_5$, $\M(r)$ is diffeomorphic to $\C\P^2$ blown up at four points.

\end{itemize}

If $r \in \mathcal{C}_1$ then $\M(r)$ is symplectomorphic to $(\C \P^2, 2\gamma \omega_{FS})$ 
and thus its Gromov width is $2 \pi \gamma=2 \pi (r_1 + \cdots +r_4 - r_5) $ as we had shown in Section \ref{section projective}.
The moment map image for the caterpillar bending action on $\M(r)$, $r \in \mathcal C_1$, is presented on Figure \ref{figure case a}.
\begin{figure}
 \psfrag{r_3}{\footnotesize{$r_3$}}
 \psfrag{r_1+r_2}{\footnotesize{$r_1+r_2$}}
 \psfrag{r_3+r_1+r_2}{\footnotesize{$r_3+r_1+r_2$}}
 \psfrag{r_4+r_5}{\footnotesize{$r_4+r_5$}}
 \psfrag{r_5-r_4}{\footnotesize{$r_5-r_4$}}
 \psfrag{A}{\footnotesize{$A$}}
   \psfrag{B}{\footnotesize{$B$}}
     \psfrag{C}{\footnotesize{$C$}}
       \psfrag{D}{\footnotesize{$D$}}
       \psfrag{E}{\footnotesize{$E$}}
       \psfrag{F}{\footnotesize{$F$}}
 \includegraphics[width=9cm]{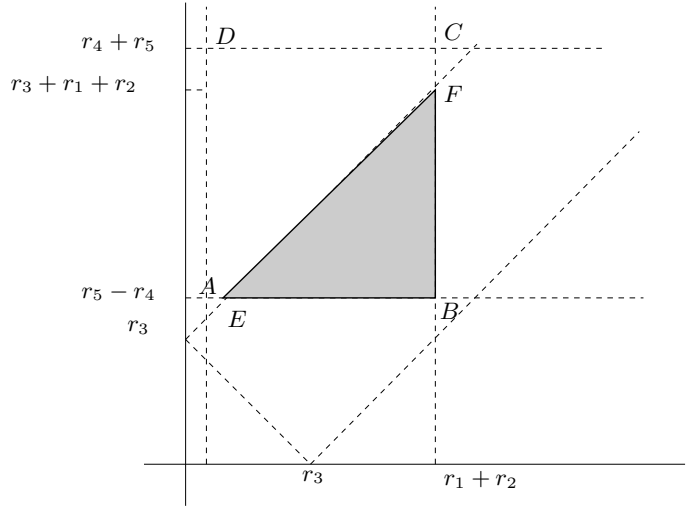}
 \caption{Moment image of $\M(r)$ for $r \in \mathcal C_1$}
\label{figure case a}
 \end{figure}

We now concentrate on the remaining chambers $\mathcal{C}_2,\ldots, \mathcal{C}_6$. Therefore for the next two subsections we assume that 
$\{1,5\}$ is short, i.e. 
$$\min \{2 r_1,\, r_1 + \cdots +r_4 - r_5\}= 2r_1 .$$

\subsection{Lower bounds}\label{section 5gons lower}
To prove the lower bound for the Gromov width of $\M(r)$ ($r$ generic) we assume
 that $r$ is ordered non-decreasingly.
Under this assumption, 
$$\min \{2 r_j,\, \big(\sum_{i \neq j} r_i\big) - r_j\,|\,
j=1,\ldots,5\} = \min \{2 r_1,\, r_1 + \cdots +r_4 - r_5\}=2r_1. $$ 
As before, consider the bending action along the caterpillar system of diagonals and let $\Delta$ be the image $(d_1,d_2)(\M(r)).$
There always exists a 
``horizontal''
segment in $\Delta$ of length $2r_1$, as we show in the next Lemma.
\begin{lemma}\label{lem:d_1}
 Let $r \in \R_+^5$ be generic, ordered non-decreasingly and such that $\{1,5\}$ is short.
 Then there exists
 $d_2^o$ s.t. $\Delta \cap \{(d_1,d_2) \in\R^2 \mid d_2= d_2^o \}$ has length
$2r_1$. 
\end{lemma}

\begin{proof}
 It is clear that there exists $d_2^o$ s.t. $\Delta \cap \{(d_1,d_2) \in\R^2 \mid d_2= d_2^o \}$ has length
$2r_1$ if and only if the triples $(d_2^o,r_4,r_5)$,$(r_1 + r_2, r_3, 
d_2^o)$ and 
$(r_2 - r_1, r_3, d_2^o)$ satisfy the triangle inequalities:
\begin{displaymath}
\left\{
 \begin{array}{r}
  r_4\leq r_5 + d_2^o\\
  r_5 \leq r_4 + d_2^o\\
  d_2^o \leq r_4 + r_5
 \end{array}
\right.
\quad
\text{and}
\quad
\left\{
 \begin{array}{r}
  r_1 + r_2\leq r_3 + d_2^o\\
  r_3 \leq r_1 + r_2 + d_2^o\\
  d_2^o \leq r_1 + r_2 + r_3
 \end{array}
\right.
\quad
\text{and}
\quad
\left\{
 \begin{array}{r}
  r_2 - r_1\leq r_3 + d_2^o\\
  r_3 \leq r_2 - r_1 + d_2^o\\
  d_2^o \leq r_2 - r_1 + r_3.
 \end{array}
\right.
\end{displaymath}
The last two sets of inequalities are verified if and only if
\begin{align*}
d_2^o \in &[ \mid r_1 + r_2 -r_3 \mid, r_1 + r_2 + r_3] \cap [ r_3 - r_2 +r_1, - 
r_1 + r_2 + r_3] \\
= & [ r_3 - r_2 +r_1,  - r_1 + r_2 + r_3] \neq \emptyset,
\end{align*}
while the first set gives the condition
$ r_5 -r_4 \leq d_2^o \leq r_4 +r_5.$
The intersection 
$$ [ r_3 - r_2 +r_1,  - r_1 + r_2 + r_3] \cap [r_5 -r_4, r_4 +r_5]$$
is non empty if and only if 
\begin{align}\label{eq:cases}
 r_3+r_2-r_1 \geq r_5 -r_4\\
r_3-r_2+r_1 \leq r_5 +r_4
\end{align}
The second inequality is verified as $\{ 1,3 \}$ is short.
Adding $2\,r_1$ to both sides of (\ref{eq:cases}) and reordering the terms, one 
obtains that there exists 
$d_2^o \in [ r_3 - r_2 +r_1,  - r_1 + r_2 + r_3] \cap [r_5 -r_4, r_4 +r_5]$
if and only if $$r_1+r_2+r_3+r_4-r_5 \geq 2 r_1,$$
i.e. if $\{1,5\}$ is short.
\end{proof}

Let $l_2$ be the real-valued function defined as follows 
\begin{align*}
 l_2 (d_1) &= \min (r_5+r_4, d_1+r_3) - \text{max}(r_5 - r_4,  | d_1 - 
r_3 |)\\
&=\min \big( 2r_4, r_5+r_4 - |d_1 - r_3 
|, d_1+r_3- r_5 + r_4, 2\min(d_1, r_3) \big)
\end{align*}
For $d_1^o \in [r_2-r_1, r_1+r_2]$ the function $l_2$ measures the length of the vertical segments, 
(non-empty by the above Lemma),
$\Delta \cap \{(d_1,d_2) \in\R^2 | d_1=d_1^o \}$.

\begin{lemma}\label{lem:d_2}
Let $r \in \R_+^5$ be generic, ordered non-decreasingly and such that $\{1,5\}$ is short.
Then
 there exists $d_1^o \in [r_2 - r_1, r_2 + r_1]$ such that $l_2 (d_1^o) \geq 2r_1$. 
\end{lemma}

\begin{proof}
 We need to find $d_1^o$ such that $$\min \big( 2r_4, r_5+r_4 - | d_1^o - r_3 
|, d_1^o+r_3- r_5 + r_4, 2\min(d_1^o, r_3) \big) \geq 2r_1.$$
If $r_3 < r_1+r_2$, then we find $d_1^o$ satisfying not only the above inequality
but also $d_1^o \geq r_3$. In fact, under this condition, the only relevant inequalities are
 \begin{equation}
   \begin{array}{rl}
   r_1+r_2\geq &d_1^o \geq r_3\\
- 2r_1 + r_3+r_4+r_5   \geq &d_1^o\\
 &d_1^o \geq 2r_1-r_3-r_4+r_5\\
  \end{array}
 \end{equation}
 Hence, any choice of 
 $$d_1^o \in [ 2r_1 -r_3 -r_4 +r_5, r_5 + r_4 +r_3 -2r_1] \cap [r_3,r_1+r_2]
\neq \emptyset$$ is such that
 $l_2 (d_1^o) \geq 2r_1$.
 The above intersection is non-empty as $r_1+r_2>2r_1 -r_3 -r_4 +r_5$ and $r_5 + r_4 +r_3 -2r_1>r_3$. 
 
 On the other hand, if $r_3 \geq r_1+r_2$, one can take $d_1^o =r_1 +r_2 $. 
 Then
 $$l_2(d_1^o)=\min \big( 2r_4, r_5+r_4 -r_3+r_1+r_2 , r_1 +r_2 +r_3- r_5 + r_4, 2(r_1 +r_2) \big)$$
 and
 $l_2 (d_1^o) \geq 2r_1$
becomes 
  \begin{equation*}
   \begin{array}{l}
 r_5+r_4 -  r_3 +r_1+r_2 \geq 2r_1\\
  r_1+r_2+r_3- r_5 + r_4 \geq 2r_1.\\
     \end{array}
 \end{equation*}
 The latter holds by assumption and implies the first one.
\end{proof}

\begin{proposition}\label{5gon Lower bound} (Lower bound)
Let $r \in \R_+^5$ be generic, ordered non-decreasingly and such that $\{1,5\}$ is short.
Then the Gromov width of $\M(r)$ is at least
$ 4\pi r_1.$
\end{proposition}

\begin{proof}
Let $d_1^o, d_2^o$ be as in Lemmas \ref{lem:d_1} and \ref{lem:d_2}.
 Then $( d_1^o, d_2^o)\in \Delta$ is a center of a diamond-like shape $\underline{\Diamond}^2(2r_1)$ fully
contained in $\Delta$.
Hence the result follows by Theorem \ref{prop diamondlike}.
\end{proof}

\subsection{Upper bounds} \label{section 5gons upper}

We now focus on finding a sharp upper bound for the Gromov width of $\M(r)$,
with $r \in \mathcal C_i$, $i=2, \ldots,6$. 
\begin{proposition}\label{5gon Upper bound} (Upper bound)
 Let $r \in \R_+^5$ be generic, ordered non-decreasingly and such that $\{1,5\}$ is short.
Then the Gromov width of $\M(r)$ is at most
$ 4\pi r_1.$
\end{proposition}

\begin{proof}
We analyze each chamber $\mathcal C_i$, $i=2, \ldots, 6$ separately.

$\boxed{r \in \mathcal C_2.}$ This chamber is non empty, for example $r=(1,2,3,4,7)\in  \mathcal C_2$. 
Note that $r_4 <r_5$ as $\{3,5\}$ is long while $\{3,4\}$ is
short, hence $d_2 \neq 0$ on $\M(r)$. 
Similarly, $r_1 \neq r_2$ because $\{2,5\}$ is long while $\{1,5\}$ is
short and so $d_1 \neq 0$ on $\M(r)$. Therefore $\M(r)$ equipped with the caterpillar bending action is a toric manifold.
 The moment map image is presented in Figure \ref{figure case b}, and is determined by
the normals and scalars (cf. \eqref{possible normals})
$$ \begin{array}{ll}
u_1= (0,1), &\lambda_1= r_5-r_4,\\
u_3= (-1,0), &\lambda_3= -(r_1+ r_2),\\
u_5= (1,-1), &\lambda_5= - r_3,\\
u_6=(1,0), & \lambda_6 = r_2-r_1.
   \end{array}$$
As there exists a monotone polytope with the above set of normals to the facets, 
the polygon space $\M(r)$ is Fano.
Note that $u_3+u_6=0$ and therefore Theorem \ref{theorem LuFanoUpperBounds} of Lu implies that the Gromov width of $\M(r)$ 
cannot be greater than
$-2 \pi ( \lambda_3+\lambda_6)=4 \pi r_1.$
\begin{figure}
\psfrag{r_2-r_1}{\footnotesize{$r_2-r_1$}}
\psfrag{r_1+r_2}{\footnotesize{$r_1+r_2$}}
\psfrag{r_3}{\footnotesize{$r_3$}}
\psfrag{r_5-r_4}{\footnotesize{$r_5-r_4$}}
\psfrag{r_4+r_5}{\footnotesize{$r_4+r_5$}}
\psfrag{u3}{\footnotesize{$u_3$}}
\psfrag{u6}{\footnotesize{$u_6$}}
 \includegraphics[width=10cm]{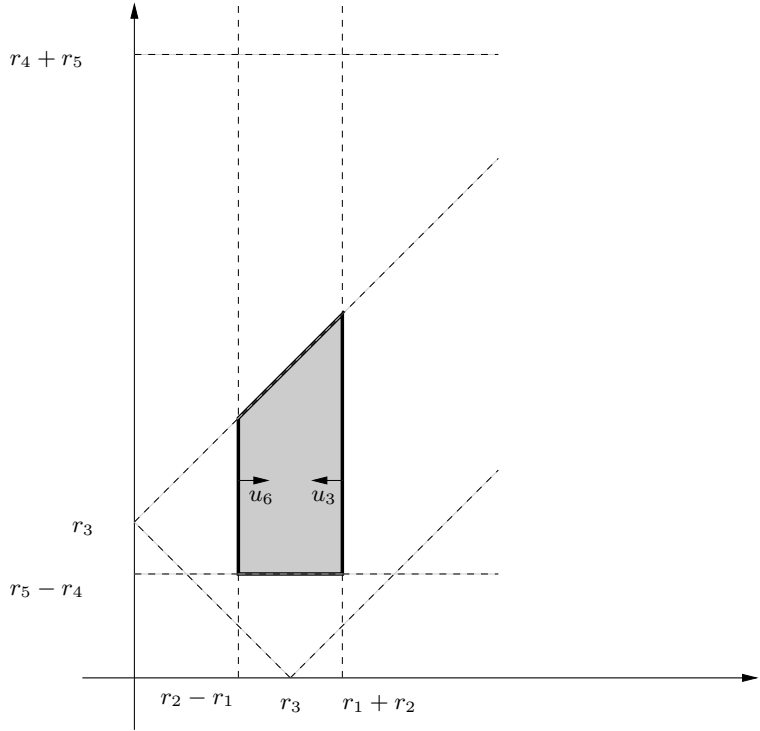}
 \caption{Moment polytope for $\M(r)$, $r \in \mathcal C_2$}
\label{figure case b}
 \end{figure}

$\boxed{r \in \mathcal C_3.}$
Example: $r= (1,2,5,6,7)$. 
For some $r$ in this chamber it might happen that $r_1=r_2$ or $r_4=r_5$, in which case $\M(r)$ would not be toric with
respect to the standard bending action. However we always have $r_2 \neq 
r_3$ because $\{2,4\}$ is short while $\{3,4\}$ is long. That also implies that 
$r_1 \neq r_5$. 
Hence if we reshuffle the length vector to $(r_2,r_3,r_4,r_1,r_5)$, then the
diagonals $d_1$ and $d_2$ never vanish on $\M(r_2,r_3,r_4,r_1,r_5)$.
Therefore the manifold $\M(r_2,r_3,r_4,r_1,r_5)$ together with 
the caterpillar bending action is a toric manifold. As it is symplectomorphic to  
$\M(r_1,r_2,r_3,r_4,r_5)$, they have the same Gromov width. To 
establish the upper bound of the Gromov width we work with the toric manifold 
$\M(r_2,r_3,r_4,r_1,r_5)$. The moment map image is always a rectangle, as presented on Figure 
\ref{figure case c}, therefore $\M(r_2,r_3,r_4,r_1,r_5)$ is Fano. 
As $u_1+u_4=0$, using Theorem \ref{theorem LuFanoUpperBounds} we obtain the upper bound of 
$4 \pi r_1$. 

\begin{figure}
\psfrag{r3-r2}{\footnotesize{$r_3-r_2$}}
\psfrag{r3+r2}{\footnotesize{$r_3+r_2$}}
\psfrag{r4}{\footnotesize{$r_4$}}
\psfrag{r5-r1}{\footnotesize{$r_5-r_1$}}
\psfrag{r5+r1}{\footnotesize{$r_5+r_1$}}
\psfrag{u1}{\footnotesize{$u_1$}}
\psfrag{u2}{\footnotesize{$u_4$}}
 \includegraphics[width=10cm]{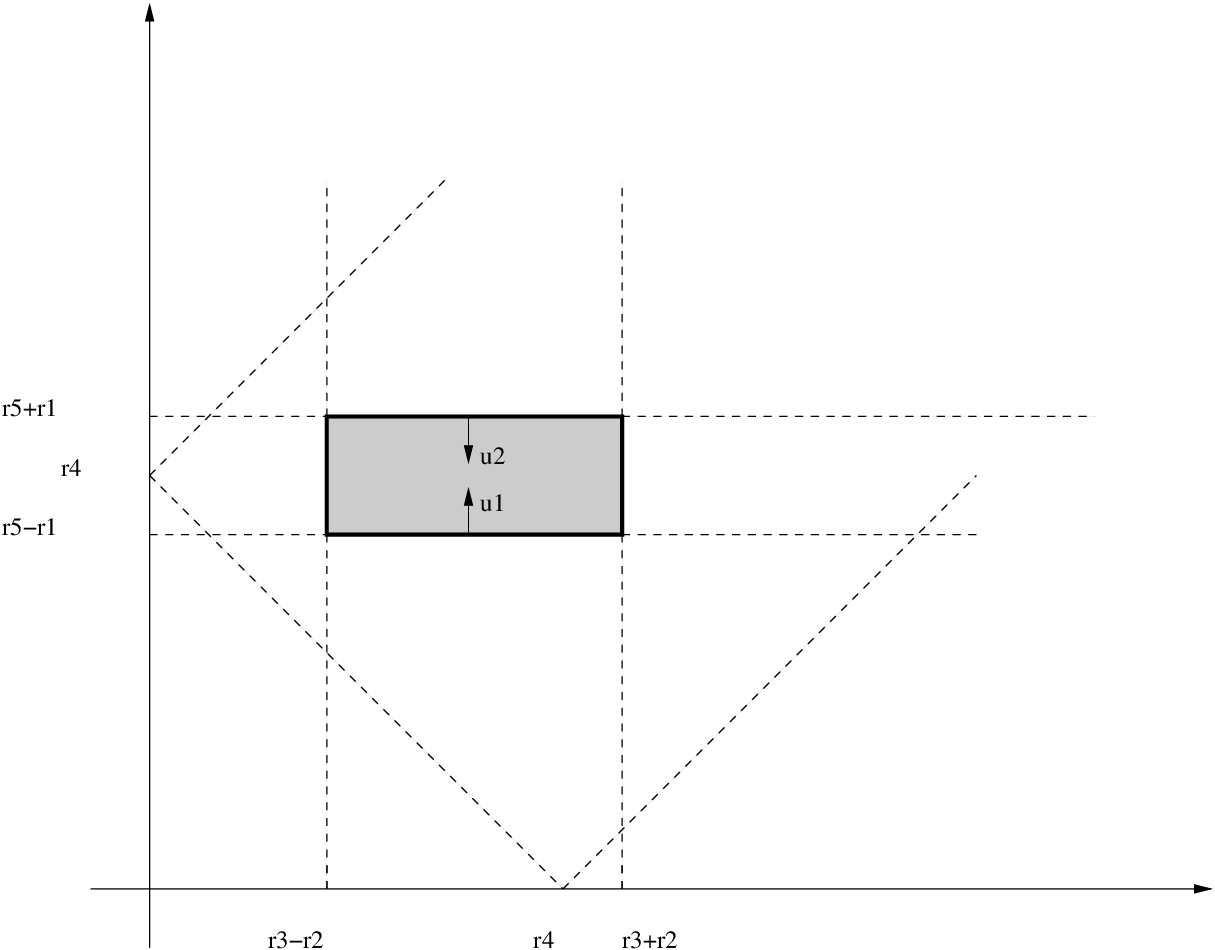}
 \caption{Moment polytope of $\M(r_2,r_3,r_4,r_1,r_5)$ for $r \in \mathcal C_3$}
\label{figure case c}
 \end{figure}

$\boxed{r \in \mathcal C_4.}$
Example: $r=(2,3,4,6,8)$.
It might happen that $r_1=r_2$. However, $r_2 \neq r_3 $ because $\{2,5\}$ is 
short while $\{3,5\}$ is long and
$r_4 \neq r_5 $ because $\{3,4\}$ is short while $\{3,5\}$ is long.
Hence the caterpillar bending action is toric on $ \M(r_2, r_3, r_1, r_4, r_5)$, 
with associated moment map image as in Figure \ref{fig:caseD}.
\begin{figure}[htbp]
\begin{center}
\psfrag{+}{\footnotesize{$r_3+r_2$}}
\psfrag{r1}{\footnotesize{$r_1$}}
\psfrag{r5-r2}{\footnotesize{$r_5+r_4$}}
\psfrag{r5+r2}{\footnotesize{$r_5-r_4$}}
\psfrag{u5}{\footnotesize{$u_5$}}
\psfrag{u2}{\footnotesize{$u_2$}}
\psfrag{A}{\footnotesize{$A$}}
\includegraphics[width=8cm]{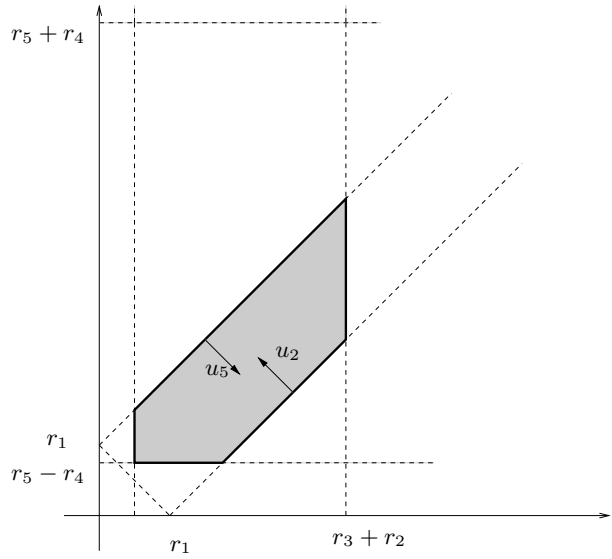}
\caption{Moment polytope of $\M(r_2,r_3,r_1,r_4,r_5)$ with $r \in \mathcal C_4$.}
\label{fig:caseD}
\end{center}
\end{figure}
$\M(r)$ is Fano and
applying Lu's Theorem \ref{theorem LuFanoUpperBounds}, we get the upper bound of $4 \pi r_1$ (relevant facet normals are $u_2$ and $u_5$).

$\boxed{r \in \mathcal C_5.}$ 
Example: $r=(2,3,3,4,5)$.
It might happen that $r_1=r_2$ and $r_4=r_5$. However
$r_3 \neq r_4 $ because $\{3,5\}$ is short while $\{4,5\}$ is long.
This also implies that $r_2 \neq r_5.$
Hence $\M(r_3, r_4, r_1, r_2, r_5)$ together with the caterpillar bending action is 
a  toric manifold with moment image as in Figure \ref{fig:caseE}.
\begin{figure}[htbp]
\begin{center}
\psfrag{r1}{\footnotesize{$r_1$}}
\psfrag{r5-r2}{\footnotesize{$r_5-r_2$}}
\psfrag{u5}{\footnotesize{$u_5$}}
\psfrag{u2}{\footnotesize{$u_2$}}
\psfrag{++}{\footnotesize{$r_5+r_2$}}
\psfrag{2-++}{\footnotesize{$-2r_1+r_2+r_5$}}
\psfrag{-}{\footnotesize{$r_4-r_3$}}
\psfrag{+}{\footnotesize{$r_4+r_3$}}
\includegraphics[width=9cm]{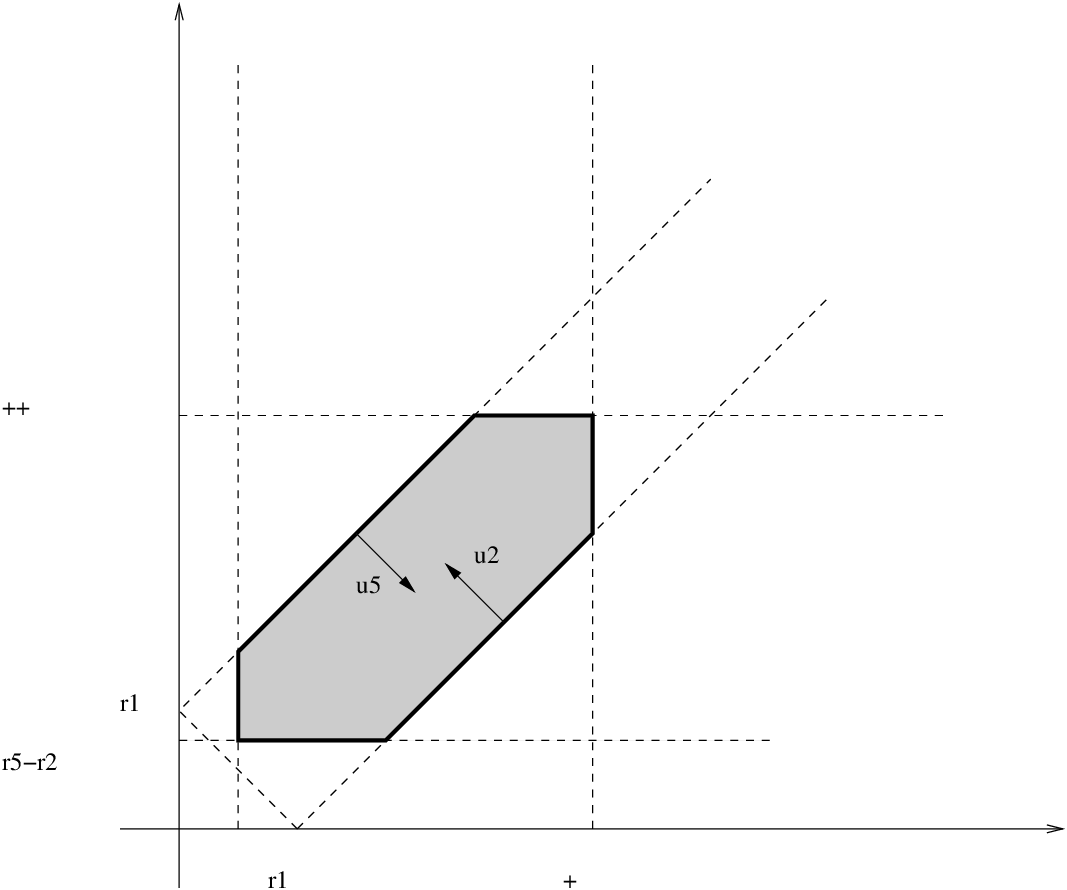}
\caption{Moment polytope of $\M(r_3,r_4,r_1,r_2,r_5)$ with $r \in \mathcal C_5$.}
\label{fig:caseE}
\end{center}
\end{figure}
Applying Lu's Theorem \ref{theorem LuFanoUpperBounds}, with relevant facet normals $u_2$ and $u_5$, we obtain the upper bound of
 $4 \pi r_1$ for the Gromov width of $\M(r)$.

$\boxed{ r \in \mathcal C_6.}$
Example: $r=(3,4,5,5,6)$.
This chamber contains some length vectors $r$ of the type
\begin{itemize}
\item[(i)] $r_1=r_2=r_3=r_4=r_5$ (equilateral case),
\item[(ii)]  $r_1 <r_2 =r_3=r_4=r_5$, example $r=(1,2,2,2,2)$,
\item[(iii)] $r_1 =r_2 =r_3=r_4<r_5$, example $r=(2,2,2,2,3)$
\end{itemize}
which are not toric for any bending action, even after reshuffling the edges.
Note however that $r=(r_1,r_1,r_1,r_1,r_5)$ is in the chamber $\mathcal C_1$ if $r_5>2r_1$ and the corresponding manifold is $\C \P^2$, hence it is toric.
It is shown in \cite{HKequilateral}, that in the equilateral case it is impossible to equip $\M(r)$ with a toric action.
For any $r \in \mathcal C_6$ not of the type (i),(ii), nor (iii), either $\M(r)$ is toric with respect to the
caterpillar bending action or 
can be equipped with a toric action by using the caterpillar bending action induced from $\M(\sigma(r))$ 
for a suitable permutation $\sigma \in S_5$.
However,
there is no universal $\sigma \in S_5$ that would work for all $r$'s in this chamber, 
as it was the case for the chambers $\mathcal C_3, \mathcal C_4, \mathcal C_5$.

Our proof for the upper bounds for $r \in \mathcal C_6$ is as follows: we first prove the claim for those $r \in\mathcal C_6$ for which $\M(r)$ with the caterpillar bending action is toric, and then we use a ``Moser type" argument to extend the result to other cases. 
Assume that $\M(r)$ is toric with the caterpillar bending action. Then the moment image of $\M(r)$ is as in Figure
\ref{figure case f}.
\begin{figure}[htbp]
\begin{center}
\psfrag{A}{\footnotesize{$A$}}
\psfrag{B}{\footnotesize{$B$}}
\psfrag{C}{\footnotesize{$C$}}
\psfrag{D}{\footnotesize{$D$}}
\psfrag{r3}{\tiny{$r_3$}}
 \psfrag{r2+r1}{\footnotesize{$r_1+r_2$}}
 \psfrag{r2-r1}{\footnotesize{$r_2-r_1$}}
 \psfrag{r5+r4}{\footnotesize{$r_4+r_5$}}
 \psfrag{r5-r4}{\footnotesize{$r_5-r_4$}}
\includegraphics[width=9cm]{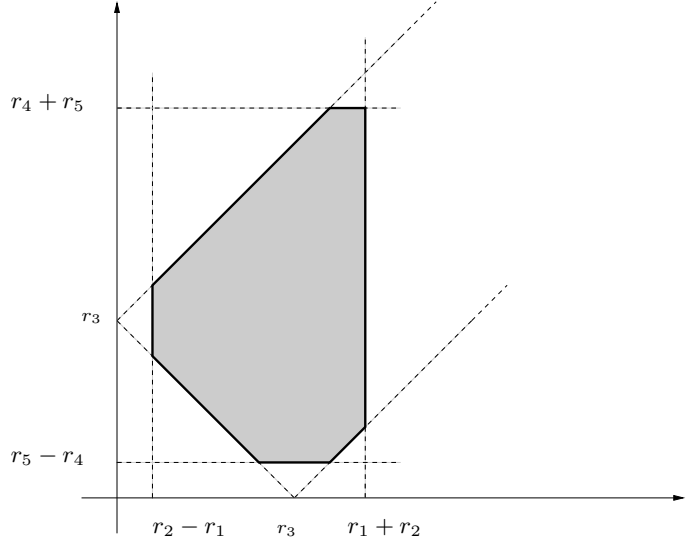}
\caption{Moment polytope of $\M(r) $ with $r \in \mathcal C_6$.}
\label{figure case f}
\end{center}
\end{figure}
From the moment polytope we can see that $\M(r)$ is not Fano 
and hence we cannot apply Lu's Theorem \ref{theorem LuFanoUpperBounds} directly. 
However $\M(r)$ is the toric blow up at three toric fixed points of the symplectic 
toric manifold $(\C \P^1 \times \C \P^1, 4r_1 \omega_{FS} \oplus 4r_4 \omega_{FS})$ corresponding, via the Delzant construction, 
to the rectangle $ABCD$.

Applying Theorem \ref{theorem LuBlowupUpperBounds} of Lu we obtain that the Gromov width of $\M(r)$ is at most 
$\Upsilon (\Sigma_{ABCD}, 2 \pi \, \varphi_{ABCD})$, which is at most $4 \pi r_1.$

Now we use the Moser method to find an upper bound for the Gromov width of the remaining cases by a 
continuity argument. We are grateful to D. Joyce for the idea of using continuity and to 
Y. Karshon for help with the details.

Consider $\M(r)$, with $r \in \mathcal C_6$ for which $r_1=r_2$ or $r_4=r_5$.
Let $\M_t$ be the family of polygon spaces $\M_t:=\M(r_1,r_2+t,r_3,r_4,r_5+t)$, for $t>0$, small enough so that $(r_1,r_2+t,r_3,r_4,r_5+t)$ is still generic.
Note that for a small positive $t$ the underlying differentiable manifold is the same for all $\M_t=(M, \omega_t)$. The length vector (which depends on $t$) encodes the different symplectic structures $\omega_t$ on the differentiable manifold $M$. 
Moreover, $\M_t$, with caterpillar bending action, is a symplectic toric manifold
and the Gromov width of $\M_t$ is $4\pi r_1$.
The key observation is that a ball of capacity bigger than $4\pi r_1$ cannot be embedded into 
$\M_0=\M(r)=(M, \omega_0)$
because given any symplectic embedding of $B_a$ into $\M_0$ 
we can always construct an embedding of $B_{a - \varepsilon}$ into $\M_t$ for $t \neq 0$ and $\varepsilon>0$ small enough, as we show below.

Take any symplectic embedding of a ball of capacity $a$, $\psi \colon ( B_a, \omega_{std}) \hookrightarrow (M, \omega_0)$. 
That is, we have a smooth map $\psi \colon B_a \rightarrow M$ such that $ \psi^* \omega_0=\omega_{std}$.
Denote $\Omega_t :=\psi^* (\omega_t)$ on $B_a$.
Following the arguments in Lemma 2.1 and Remark 2.2 of McDuff \cite{McDuff} we will use ``Moser's trick'' to construct a smooth embedding  
$$\phi_t \colon B_{a -\varepsilon} \rightarrow B_a$$
such that $$\phi_t^*(\Omega_t)=\omega_{std}.$$
Then $\psi_t:= \psi \circ \phi_t \colon B_{a- \varepsilon} \rightarrow M$ will be a symplectic embedding of $B_{a- \varepsilon}$ into $\M_t$, as $\psi_t^* (\omega_t)=\omega_{std}$.

Observe that 
\begin{equation}\label{pull back}
\phi_t^*(\Omega_t)=\omega_{std}\,\,\, \Leftrightarrow \,\,\,\frac{d}{dt}( \phi_t^* \Omega_t)=0.
\end{equation}
Let $X_t$ denote the vector field generated by the isotopy $\phi_t$, i.e. $\frac{d}{dt}\, \phi_t= X_t\circ \phi_t$. Then
 $$\frac{d}{dt}( \phi_t^* \Omega_t)=\phi^*_t (\mathfrak{L}_{X_t}\Omega_t + \frac{d}{dt} \Omega_t).$$
By the Poincar\'e Lemma (with parameters; cf \cite[Remark 2.2]{McDuff}) for $B_a$, there exist $\lambda_t$ such that $ \frac{d}{dt} \Omega_t=-d\lambda_t$. 
Therefore, using the Cartan formula we get
$$ \frac{d}{dt} \phi_t^* \Omega_t= \phi_t^*(\mathfrak{L}_{X_t}\Omega_t -d\lambda_t)=
\phi_t^*( d (\iota_{X_t}\Omega_t)+ \iota_{X_t}( d \Omega_t) - d \lambda_t)=\phi_t^* d(\iota_{X_t}\Omega_t-\lambda_t).$$
If $$ \iota_{X_t}\Omega_t=\lambda_t$$ 
then $ \frac{d}{dt} \phi_t^* \Omega_t$ is certainly $0$, and thus $\phi_t^*(\Omega_t)=\omega_{std}$ by \eqref{pull back}.
The non-degeneracy of $\Omega_t$ on $B_a$ guarantees that this equation can always be solved for $X_t$. 
For each $p \in B_a$, by integrating $X_t$ one obtains its flow $\phi_t$ defined in some neighborhood of $p$. 
The orbit $\phi_t(p)$ stays in $B_a$ for small $t$.
Given any $t>0$ we cannot guarantee that $\phi_t$ is defined on the whole $B_a$. 
However, given any $\varepsilon >0$ we can find a small $t>0$ such that 
$\phi_t(B_{a-\varepsilon}) \subset B_a$. Then the map
$$\psi \circ \phi_t \colon (B_{a-\varepsilon}, \omega_{std}) \rightarrow (M,\omega_t)$$
is a symplectic embedding. As the Gromov width of $(M,\omega_t)$ is $4\pi r_1$, we must have
$a - \varepsilon< 4 \pi r_1$ for each $\varepsilon >0$.
This proves that $a \leq 4 \pi r_1$ for all $a$ such that the ball $B_a$ of capacity $a$ symplectically embeds into $(M, \omega_0)=\M(r).$

\end{proof}

\section{Gromov width of the spaces of $6$-gons}\label{section 6gons}
In this section we analyze the Gromov width of the space of $6$-gons $\M(r)$
(as usually, $r \in \R_+^6$ is assumed to be generic and thus $\M(r_1,\ldots,r_6)$ is a smooth manifold; 
see Section \ref{section polygon spaces}).  
Recall that our Theorem \ref{theo:6gon} states that 
\begin{equation}\label{eq:min6}
 \rho(r)=2\pi \min \{2r_j,\,\big(\sum_{i \neq j} r_i \big) - r_j\,|\, j=1,\ldots,6\}
\end{equation}
is the lower bound for the Gromov width of $\M(r)$ 
and that if $\sigma \in S_6$ is such that $r_{\sigma(1)}\leq \ldots \leq r_{\sigma(6)}$ and
one of the following holds:
\begin{itemize}
 \item $\{1,2,3,4\}$ and $\{1,2,6\}$ are short for $\sigma(r)$, or
\item $\{1,2,6\}$ and $\{4,6 \}$ are long for $\sigma(r)$, or
\item $\{ 5,6 \}$ and $\{ 2,3,6\}$ are short for $\sigma(r)$
\end{itemize}
then the above formula is exactly the Gromov width of $\M(r)$.
As $\M(r)$ and $\M(\sigma(r))$ are symplectomorphic for each permutation $\sigma \in S_6$, 
we continue to work with the assumption that $r_1 \leq \ldots \leq r_6$, 
Then the value  of \eqref{eq:min6} is $2 \pi \min\{2r_1, (r_1+\ldots+r_5)-r_6\} $.
In this section we use the bending action along the system of diagonals 
as in Figure \ref{fig:6gonsystem}.
\begin{figure}[htbp]
\begin{center}
\psfrag{1}{$e_1$}
\psfrag{2}{$e_2$}
\psfrag{3}{$e_3$}
\psfrag{4}{$e_4$}
\psfrag{5}{$e_5$}
\psfrag{6}{$e_6$}
\psfrag{d1}{$d_1$}
\psfrag{d2}{$d_2$}
\psfrag{d3}{$d_3$}
\includegraphics[width=6cm]{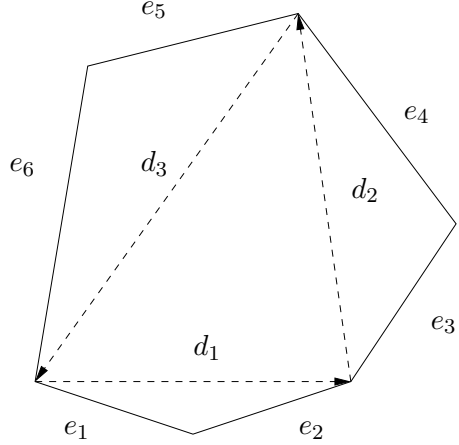}
\caption{System of diagonals.}
\label{fig:6gonsystem}
\end{center}
\end{figure}
The functions $d_i \colon \M(r) \rightarrow \R$, $i=1,\,2,\,3$, denote the lengths of the respective diagonals. 
They are continuous on the whole $\M(r)$ and smooth on the dense subset $\{d_i \neq 0\,|\,i=1,\,2,\,3\} \subset \M(r)$.
This subset is equipped with the toric bending action for which the function $(d_1,\,d_2,\,d_3)$, restricted to $\{d_i \neq 0\,|\,i=1,\,2,\,3\}$, 
is a moment map.
The image $\Delta$ of the (continuous) map $(d_1,d_2,d_3) \colon \M(r) \rightarrow \R^3$ is the region in $\R^3$ bounded by 
the triangle inequalities:
\begin{equation}\label{6gonmomentpolytope}
\begin{aligned}
  r_2-r_1 &\leq d_1 \leq r_2+r_1,\\ 
r_4-r_3 &\leq d_2 \leq r_3 +r_4,\\
r_6-r_5 &\leq d_3 \leq r_5+r_6,\\
|d_1-d_2| &\leq d_3 \leq d_1 +d_2
\end{aligned}
\end{equation}
By a slight abuse of notation we denote the coordinates of $\R^3$ also by $d_1,\,d_2,\,d_3$.
Let $C$ be the cuboid of points 
satisfying the first three pairs of inequalities \eqref{6gonmomentpolytope}, and let $H_j^+$ be 
the affine half-space 
\begin{equation}\label{eq:hj+}
 H_j^+ =\big \{\sum_{i=1}^{3} d_i \geq 2 d_j \big \},
\end{equation} bounded by an affine 
hyperplane $H_j:=\{\sum_{i=1}^{3} d_i = 2 d_j\}$, $j=1,2,3$. Then $$\Delta= C \cap 
\bigcap_{j=1}^{3}H_j^+.$$

If $\{1,6\}$ is long, i.e. $\gamma:=r_1+\ldots+r_5-r_6<2r_1$ then $\M(r)$ is symplectomorphic to $\C \P^3$ and its Gromov width is $2\pi \gamma$ as we showed in Section \ref{section projective}. 
One can also see it directly here by observing that $\Delta$ is a simplex with vertices \begin{align*}
v_3&=(r_2+r_1,r_3 +r_4,r_6-r_5),\\
p_1&=v_3-\gamma(1,0,0)=(r_6-r_5-r_3-r_4,r_3 +r_4,r_6-r_5),\\
p_2&=v_3-\gamma(0,1,0)=(r_2+r_1,r_6-r_5-r_1-r_2,r_6-r_5),\\
p_3&=v_3+\gamma(0,0,1)=(r_2+r_1,r_3 +r_4,r_2+r_1+r_3 +r_4),
\end{align*}
which is fully contained in $\R_+^3$, see Figure \ref{fig:simplex}.
\begin{figure}[htbp]
\begin{center}
\psfrag{p0}{\footnotesize{$v_3$}}
\psfrag{p1}{\footnotesize{$p_1$}}
\psfrag{p2}{\footnotesize{$p_2$}}
\psfrag{p3}{\footnotesize{$p_3$}}
\psfrag{d1}{\footnotesize{$d_1$}}
\psfrag{d2}{\footnotesize{$d_2$}}
\psfrag{d3}{\footnotesize{$d_3$}}
\includegraphics[width=6cm]{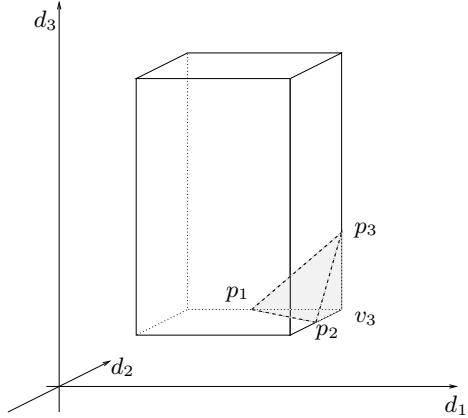}
\caption{The moment map image $\Delta=C \cap H_3^+= C \cap \bigcap_{j=1}^{3}H_j^+$.}
\label{fig:simplex}
\end{center}
\end{figure}

Now we concentrate on the cases when $\{1,6\}$ is short, that is, we work with the assumption 
\begin{equation}\label{assumption 6gon}
\gamma=r_1+\ldots+r_5-r_6>2r_1
\end{equation}
and thus the expected Gromov width is $4\pi r_1$.

\subsection{Lower bounds} To determine the lower bound for the Gromov width of
$\M(r)$ we can fit a diamond-like open subset $\underline{ \Diamond}^n(2 r_1)$
in $\Delta$. As $\Delta$ is convex by construction, it is sufficient to prove 
that in $\Delta$ there are segments parallel to the $d_1, d_2, d_3$ axis respectively, 
each of length at least 2$r_1$ and intersecting at a point $(d_1^o, d_2^o, d_3^o) \in \Delta$.
\begin{lemma}\label{lemma goodd2d3}
There are values $d_2^o,d_3^o$ such that the interval $\{(t,d_2^o,d_3^o):\,r_2-r_1 \leq 
t \leq r_1+r_2\}$ of length $2r_1$ is fully contained in $\Delta$.
\end{lemma}
\begin{proof}
 We need to show that there are values $d_2^o,d_3^o$ such that the inequalities 
\eqref{6gonmomentpolytope} are satisfied and that one can build triangles with 
edge lengths $(r_2-r_1,d_2^o,d_3^o)$,  $(r_2+r_1,d_2^o,d_3^o)$. That is, we need to 
show that the two regions presented in Figure \ref{fig:6gonfitarm} have 
non-empty intersection.
\begin{figure}[htbp]
\begin{center}
\psfrag{a}{\tiny{$r_2-r_1$}}
\psfrag{b}{\tiny{$r_2+r_1$}}
\psfrag{c}{\tiny{$r_4-r_3$}}
\psfrag{d}{\tiny{$r_4+r_3$}}
\psfrag{e}{\tiny{$r_6-r_5$}}
\psfrag{f}{\tiny{$r_6+r_5$}}
\includegraphics[width=10cm]{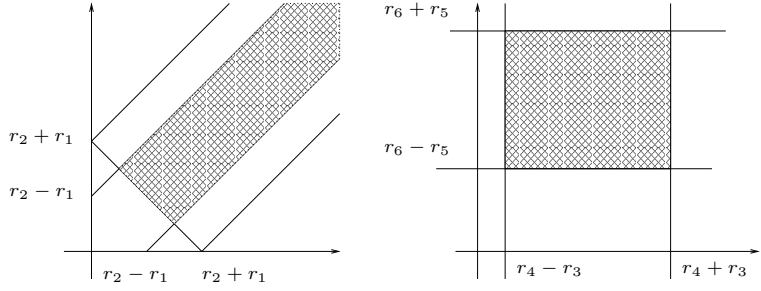}
\caption{Conditions on $d_2^o,d_3^o$.}
\label{fig:6gonfitarm}
\end{center}
\end{figure}
Note that this intersection is not empty if and only if 
\begin{displaymath}
 \begin{cases}
  r_2-r_1+r_5+r_6 & \geq r_4-r_3,\\
  r_2-r_1+r_3+r_4 & \geq r_6-r_5.
 \end{cases}
\end{displaymath}
This is equivalent to $r_2+r_3+r_4+r_5 \geq r_1+r_6$ and follows from our assumption \eqref{assumption 6gon}.
\end{proof}

Define functions $l_2,l_3 \colon \R^2 \rightarrow \R$, and $c \colon \R^3 
\rightarrow \R$ by
\begin{align*}
 l_2(d_1,d_3)&=\min(r_3+r_4, d_1+d_3) - \max (r_4-r_3, |d_1-d_3|)\\
 &=\min (2r_3, r_3+r_4 -|d_1-d_3|, d_1 +d_3-r_4+r_3, 2 
\min(d_1,d_3)),\\
l_3(d_1,d_2)&=\min(r_5+r_6, d_1+d_2) - \max (r_6-r_5, |d_1-d_2|)\\
&=\min (2r_5, r_5+r_6 -|d_1-d_2|, d_1 +d_2-r_6+r_5, 2 
\min(d_1,d_2)),\\
c(d_1,d_2,d_3)&= \min (l_2(d_1,d_3),l_3(d_1,d_2))
\end{align*}
Note that if $\Delta$ and the line $\{(d_1,d_2,d_3) \in \R^3 \mid d_1=d_1^o,d_3=d_3^o\}$
intersect non--trivially, then they intersect in an interval of length 
$l_2(d_1^o,d_3^o)$. Similarly for $l_3$.

\begin{lemma}\label{lemma crossfits}
 There exist $d_2^o,d_3^o$ as in Lemma \ref{lemma goodd2d3} and $d_1 \in 
(r_2-r_1,r_2+r_1)$ such that $c(d_1,d_2^o,d_3^o) \geq 2r_1$. 
\end{lemma}
\begin{proof}

We need to find $d_1,d_2^o,d_3^o$ such that $d_2^o,d_3^o$ are from Lemma \ref{lemma goodd2d3}, i.e., they are in the intersection of the 
two regions presented in Figure \ref{fig:6gonfitarm}, and that 
\begin{align*}
 \min (2r_3, r_3+r_4 -|d_1-d_3^o|, d_1 +d_3^o-r_4+r_3, 2 \min(d_1,d_3^o))& \geq 2r_1\\
\min (2r_5, r_5+r_6 -|d_1-d_2^o|, d_1 +d_2^o-r_6+r_5, 2 \min(d_1,d_2^o))& \geq 2r_1.
\end{align*}
We show that there exist $d_1^o,d_2^o,d_3^o$ satisfying not 
only the above conditions but also: $d_2^o,d_3^o \geq d_1^o\geq r_1$. 
The only non-trivial conditions in the above inequalities are 
\begin{align*}
r_3+r_4 +d_1-d_3^o & \geq 2r_1,\\
d_1 +d_3^o-r_4+r_3 & \geq 2r_1,\\
r_5+r_6 +d_1-d_2^o& \geq 2r_1,\\
d_1 +d_2^o-r_6+r_5& \geq 2r_1.
\end{align*}
This gives the following conditions on $d_2^o,d_3^o$
\begin{align*}
r_3+r_4 +d_1- 2r_1& \geq  d_3^o,\\
d_3^o& \geq 2r_1-r_3+r_4 -d_1 ,\\
r_5+r_6 +d_1-2r_1& \geq d_2^o ,\\
d_2^o& \geq 2r_1-r_5+r_6-d_1.
\end{align*}
Combining the above with the conditions in Figure \ref{fig:6gonfitarm} we obtain 
the 
intersection of the two regions in Figure \ref{fig:6gonfitcross},
\begin{figure}[htbp]
\begin{center}
\psfrag{a}{\tiny{$r_2-r_1$}}
\psfrag{b}{\tiny{$r_2+r_1$}}
\psfrag{c}{\tiny{$A_1$}}
\psfrag{d}{\tiny{$A_2$}}
\psfrag{e}{\tiny{$B_1$}}
\psfrag{f}{\tiny{$B_2$}}
\includegraphics[width=10cm]{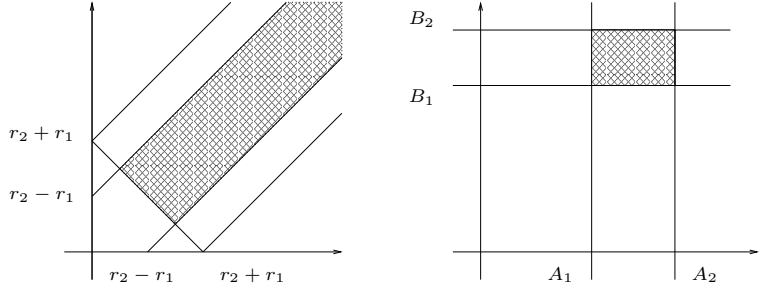}
\caption{Conditions on $d_2^o,d_3^o$.}
\label{fig:6gonfitcross}
\end{center}
\end{figure}
where 
\begin{align*}
A_1&=\max(r_4-r_3,2r_1-r_5+r_6-d_1,d_1),\\
A_2&=\min(r_4+r_3,r_5+r_6 +d_1-2r_1),\\
B_1&=\max(r_6-r_5,2r_1-r_3+r_4 -d_1, d_1),\\
B_2&=\min(r_6+r_5,r_3+r_4 +d_1- 2r_1).
\end{align*}
Note that $A_1, A_2, B_1, B_2$ are functions of $d_1$, 
which in turn satisfies $r_1+r_2 \geq d_1 \geq r_1$.
The intersection of the regions in Figure \ref{fig:6gonfitcross} is not empty if and only if 
\begin{displaymath}
 \begin{cases}
  r_2-r_1+B_2 & \geq A_1,\\
  r_2-r_1+A_2 & \geq B_1.
 \end{cases}
\end{displaymath} that is
\begin{displaymath}
 \begin{cases}
  r_2-r_1+\min(r_6+r_5,r_3+r_4 +d_1- 2r_1) & \geq 
\max(r_4-r_3,2r_1-r_5+r_6-d_1,d_1),\\
  r_2-r_1+\min(r_4+r_3,r_5+r_6 +d_1-2r_1) & \geq 
\max(r_6-r_5,2r_1-r_3+r_4 
-d_1, 
d_1).
 \end{cases}
\end{displaymath}
Most of these inequalities follow easily from the assumptions that $r_1 \leq 
\ldots \leq r_6$, $\max(r_1,r_2-r_1) \leq d_1\leq r_1+r_2$ and $ 
\min\{2r_1, 
(r_1+\ldots+r_5)-r_6\} =2r_1$, so $r_2+r_3+r_4 +r_5 \geq r_1+r_6$. The relevant ones are 
\begin{align*}
    2d_1 & \geq 5r_1-r_2-r_3-r_4-r_5+r_6,\\
2d_1 & \geq 5 r_1-r_2-r_3+r_4-r_5 -r_6.
\end{align*}
The second inequality follow from the first one. Thus the only relevant 
condition is 
$$2d_1 \geq 5r_1-r_2-r_3-r_4 -r_5+r_6.$$
To ensure the existence of $d_1 \in [\max(r_1,r_2-r_1) , r_1+r_2]$ satisfying 
the above condition it suffices to ensure that $2r_1+2r_2 \geq 5r_1-r_2-r_3-r_4 
-r_5+r_6$ i.e. that
$$3r_2 +r_3+r_4 +r_5\geq 3r_1+r_6.$$
This holds by assumptions, as
\begin{align*}
    2r_2& \geq 2r_1,\\
r_2+r_3+r_4+r_5  & \geq  r_1+r_6.
\end{align*}
\end{proof}

\begin{proposition}\label{proposition 6lowerbounds}
Assume $r \in \R_+^6$ is generic, ordered non-decreasingly and that 
 $\gamma \geq 2r_1$. Then the Gromov width of $\M(r_1,\ldots,r_6)$ is at 
least 
$4 \pi r_1$.
\end{proposition}
\begin{proof}
 Take $(d_1^o,d_2^o,d_3^o)$ from Lemma \ref{lemma crossfits}. Then the sets 
 \begin{align*}
E_1&:=\{d_2=d_2^o,d_3=d_3^o\} \cap \Delta,\\  E_2&:=\{d_1=d_1^o,d_3=d_3^o\} \cap \Delta,\\ E_3&:=\{d_1=d_1^o,d_2=d_2^o\} \cap \Delta
\end{align*}
are intervals of length 
greater or equal $2r_1$ and they intersect at  $(d_1^o,d_2^o,d_3^o)$. 
Therefore their convex hull, $Conv(E_1,E_2,E_3)$ 
contains the closure of a diamond-like 
open region, $\underline{\Diamond}(2r_1)$, of size $2r_1$. 
Moreover $Conv(E_1,E_2,E_3)$ is 
contained in $\Delta$ (by convexity of $\Delta$). 
Hence it follows from Proposition \ref{prop 
diamondlike} that the Gromov width of $\M(r_1,\ldots,r_6)$ is at least 
$2 \pi \cdot 2r_1$.
\end{proof}
\subsection{Upper bounds}
We now turn to finding upper bounds for the Gromov width of $\M(r)$ in the cases when $r$ is generic and 
no maximal $r$-short index set has cardinality $1$, i.e. 
$\min\{2r_j, (\sum_{i \neq j}r_i)-r_j\,|\,j=1,\ldots,6\} = 2\min\{r_j\,|\,j=1,\ldots,6\}$.
(If such a maximal short set exists then $\M(r)$ is diffeomorphic to projective space as described in the Section \ref{section projective}.)
The goal is to show that the Gromov width cannot be greater than $4 \pi \min\{r_j\,|\,j=1,\ldots,6\}$.

For simplicity of notation we assume that the length vector $r$ is reshuffled so that
$$r_1 \leq r_2,\ r_3 \leq r_4,\ r_5\leq r_6.$$ 
This partial ordering allows us to say that, for example, the values of $d_1$ on $\M(r)$ are in the interval $[r_2-r_1, r_1+r_2]$, instead of saying $[|r_2-r_1|, r_1+r_2]$.
The image $\Delta=(d_1,d_2,d_3)(\M(r))$ is then
$$\Delta= C \cap \bigcap_{j=1}^{3}H_j^+$$
where $C$ is the cuboid of vertices $v_1, \ldots,v_8$:
\begin{align*}
v_1&=(r_2-r_1,r_4-r_3,r_6-r_5), &v_5&=(r_2-r_1,r_4-r_3,r_6+r_5),\\
v_2&=(r_2+r_1,r_4-r_3,r_6-r_5),&v_6&=(r_2+r_1,r_4-r_3,r_6+r_5),\\
v_3&=(r_2+r_1,r_4+r_3,r_6-r_5), &v_7&=(r_2+r_1,r_4+r_3,r_6+r_5),\\
v_4&=(r_2-r_1,r_4+r_3,r_6-r_5),&v_8&=(r_2-r_1,r_4+r_3,r_6+r_5),
\end{align*}
\begin{figure}[htbp]
\begin{center}
\psfrag{1}{\footnotesize{$v_1$}}
\psfrag{2}{\footnotesize{$v_2$}}
\psfrag{3}{\footnotesize{$v_3$}}
\psfrag{4}{\footnotesize{$v_4$}}
\psfrag{5}{\footnotesize{$v_5$}}
\psfrag{6}{\footnotesize{$v_6$}}
\psfrag{7}{\footnotesize{$v_7$}}
\psfrag{8}{\footnotesize{$v_8$}}
\psfrag{d1}{\footnotesize{$d_1$}}
\psfrag{d2}{\footnotesize{$d_2$}}
\psfrag{d3}{\footnotesize{$d_3$}}
\includegraphics[width=6cm]{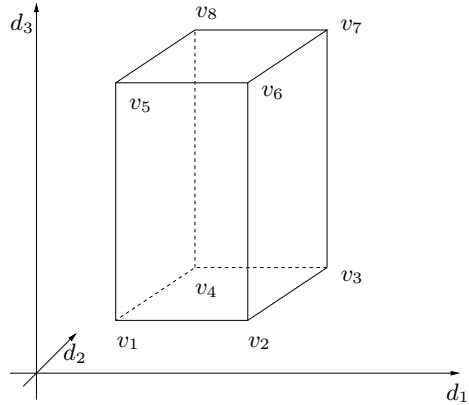}
\caption{The cuboid $C$.}
\label{fig:vertices}
\end{center}
\end{figure}
see Figure \ref{fig:vertices},
and $H_i^+$, $i=1,\,2,\,3,$ are the affine half spaces as in \eqref{eq:hj+}.
The hyperplanes $H_i$, $i=1,\,2,\,3,$ may give 
rise to the facets of $\Delta$ with inward normals $w_1=(-1,1,1)$, $w_2=(1,-1,1)$, $w_3=(1,1,-1).$

The table below collects the information, obtained by a straightforward computation, 
about when the vertices of $C$ belong to $H_i^+$ as well. 

\begin{tabular}{c|c|c|c}\label{chart}
vertex & is in $H_1^+$ if & is in $H_2^+$ if & is in $H_3^+$ if \\ \hline
$v_1$ &\{2,3,5\} is short & \{1,4,5\} is short &\{1,3,6\} is short\\ \hline
$v_2 $&\{1,2,3,5\} is short & \{4,5\} is short &\{3,6\} is short\\ \hline
$v_3 $&\{1,2,5\} is short & \{3,4,5\} is short & \{6\} is short \\ \hline
$v_4$ &\{2,5\} is short  & \{1,3,4,5\} is short &\{1,6\} is short \\ \hline
$v_5$ &\{2,3\} is short  & \{1,4\} is short  &\{1,3,5,6\} is short\\ \hline
$v_6$ &\{1,2,3\} is short  & \{4\} is short  &\{3,5,6\} is short\\ \hline
$v_7$ &\{1,2\} is short  & \{3,4\} is short  &\{5,6\} is short\\ \hline
$v_8$ & \{2\} is short  & \{1,3,4\} is short  &\{1,5,6\} is short\\ \hline
\end{tabular}
\begin{proposition}\label{6gon ub toric2}
Let $r \in \R^6_+$ be generic, ordered non-decreasingly and such that $\{1,6\}$ is short. 
Assume in addition that $\{4,6\}$ and $\{1,2,6\}$ are long. Then the Gromov width of the symplectic toric manifold $\M(r)$ is at most $4 \pi r_1$.
\end{proposition}

\begin{proof}
Reshuffle the length vector $r$ to
$$\sigma(r)=(r_1, r_4, r_2, r_5, r_3,r_6).$$
Note that $\sigma(r)$ is partially ordered and thus we can use the above table (applied to $\sigma(r)$) 
to analyze the set $\Delta= C \cap \bigcap_{j=1}^{3}H_j^+$, which is the image of $\M(\sigma(r))$ by $(d_1,d_2,d_3)$.
As $\{1,2,3,5\}$ and $\{1,2,3,4\}$ are short for $r$, 
(so $\{1,3,4,5\}$ and $\{1,2,3,5\}$ are short for $\sigma(r)$), the hyperplanes $H_1$, $H_2$ do not cut any vertex of the cuboid, 
and thus $\Delta=C \cap H_3^+$. Note that the assumption $\{1,2,6\}$  long implies that $v_1, v_5,v_6,v_8$ are not in $H_3^+$. 
The vertex $v_4$ is always in $H_3^+$ as we are assuming that $\{1,6\}$ is short. 
Depending  on whether $0$, $1$, or $2$ of the sets $\{2,6\}$ and $\{3,6\}$ are short, (corresponding to $0$, $1$, or $2$ of 
the vertices $v_2,v_7$ being in $H_3^+$), the set $\Delta$ is a simplex with $1$, $2$ or $3$ corners chopped off, 
respectively (one corner is always chopped off as $v_4$ is in $H_3^+$). 
Let $\mathcal{S}$ be the simplex bounded by the hyperplanes $H_3$, $\{d_1=r_1+r_4\}$, $\{d_2=r_2+r_5\}$ and $\{d_3=r_6-r_3\}$. 
The vertex $H_3 \cap \{d_2=r_2+r_5\} \cap \{d_3=r_6-r_3\}$ of the simplex $\mathcal{S}$ is chopped off in $\Delta$ by the hyperplane $\{d_1=r_4-r_1\}$ (as $v_4 \in H_3^+$). Let $\Delta'$ denote the simplex $\mathcal{S}$ with one corner chopped. Note that the vectors $(1,0,0)$ and $(-1,0,0)$ are among the inward normals to the facets of $\Delta'$.
The vertices $H_3 \cap \{d_1=r_1+r_4\} \cap \{d_3=r_6-r_3\}$ and $H_3 \cap  \{d_1=r_1+r_4\} \cap \{d_3=r_6+r_3\}$ of the simplex $\mathcal{S}$ may also be chopped in $\Delta$ depending on whether $\{2,6\}$ and $\{3,6\}$ are short.

If $r_1 \neq r_4$, $r_2\neq r_5$, and $r_3\neq r_6$ then the bending action on $\M(\sigma(r))$ is toric and $\Delta$ is the moment map image.
This implies that $\M(\sigma(r))$ with the bending action is 
$\C \P ^3$ blown up at $1$, $2$ or $3$  points.
In other words, it is a toric Fano manifold corresponding to the polytope $\Delta'$, or a blow up of this manifold at $1$ or $2$ toric fixed points.
Applying Theorem \ref{theorem LuFanoUpperBounds} or \ref{theorem LuBlowupUpperBounds} we get that the Gromov width of $\M(\sigma(r))$ is at most $ 2 \pi \,(r_1 +r_2 -(r_2-r_1))=4 \pi r_1.$ 
Since $\M(\sigma(r))$ and $\M(r)$ are symplectomorphic, the Gromov width of $\M(r)$ is also at most $4 \pi r_1$.

If at least one of $r_1 \neq r_4$, $r_2\neq r_5$, and $r_3\neq r_6$ is not satisfied, then the bending action is defined only on an open dense subset of  $\M(\sigma(r))$ and the above argument does not apply. In this situation, one can use Moser's trick as in the case 
of $5$-gons described in detail in Section \ref{section 5gons}.
Let 
$$\M_t(r)=\M(r_1,r_2,r_3,r_4+t,r_5+t,r_6+t).$$ 
For $t>0$ small, the polygon space $\M_t(r)$ with the bending action induced using the symplectomorphism 
$$\M(r_1,r_2,r_3,r_4+t,r_5+t,r_6+t) \simeq \M(r_1, r_4+t, r_2, r_5+t, r_3,r_6+t)$$ is toric.
Moreover, 
if $\{4,6\}$ and $\{1,2,6\}$ are long for $\M(r)$, 
then $\{4,6\}$ and $\{1,2,6\}$ are also long for $\M_t(r)$ 
 and the Gromov width of $\M_t(r)$ is $4\pi r_1$ by Proposition \ref{6gon ub toric2}.

Assume that there exists a symplectic embedding of a ball of capacity $a > 4 \pi r_1$ into $\M(r)$.
As in Section \ref{section 5gons} we use Moser's trick to show that the Gromov width of $\M(r)$ is not larger than $4\pi r_1$.
%Use Moser's trick argument to show that for $\varepsilon >0$ there exists an embedding
%of a symplectic ball of capacity $a - \varepsilon$ into $\M_t(r)$ for $t>0$ small enough (see the end 
%of Section \ref{section 5gons} for details).
%Taking $\varepsilon $ small enough so that $a - \varepsilon> 4 \pi r_1$ we obtain a contradiction.
Therefore there cannot exist an embedding of a ball of capacity $a > 4 \pi r_1$ into $\M(r)$.
\end{proof}

\begin{proposition}\label{6gon ub 3bu}
Let $r$ be generic, ordered non-decreasingly and such that 
$\{2,3,6\}$ and $\{5,6\}$ are short. Then the Gromov width of the symplectic toric manifold $\M(r)$ is at most $4 \pi r_1$.
\end{proposition}

\begin{proof}
Reshuffle the length vector $r$ to
$$\sigma(r)=(r_1, r_4, r_2, r_5, r_3,r_6).$$
Note that $\sigma(r)$ is partially ordered and thus we can use the above table (applied to $\sigma(r)$) to analyze the set 
$\Delta= C \cap \bigcap_{j=1}^{3}H_j^+$, which is the image of $\M(\sigma(r))$ by $(d_1,d_2,d_3)$.
As $\{1,2,3,4\}$ is long (by assumption), all $4$--element sets are long, and all $2$--element sets are short. Thus each hyperplane is cutting at least one vertex.
Our assumptions guarantee that each of them cuts exactly one vertex (out of the vertices $v_2$, $v_4$, $v_5$). Therefore $\Delta$ is the cuboid $C$ with three non-adjacent corners chopped off.
If $r_1 \neq r_4$, $r_2\neq r_5$, and $r_3\neq r_6$ then the bending action on $\M(\sigma(r))$ is toric, and $\M(\sigma(r))$ is symplectomorphic to the blow up of the toric Fano manifold $(\C\P^1 \times \C\P^1 \times \C\P^1, 4r_1 \omega_{FS}\oplus 4r_2 \omega_{FS}\oplus 4r_3 \omega_{FS})$ (corresponding to the cuboid $C$), at three toric fixed points. Applying Theorem \ref{theorem LuBlowupUpperBounds} we get that the Gromov width of $\M(\sigma(r))$ is at most $ 2 \pi \,(r_1 +r_2 -(r_2-r_1))=4 \pi r_1.$

If at least one of $r_1 \neq r_4$, $r_2\neq r_5$, and $r_3\neq r_6$ is not satisfied, one uses Moser's trick as above.
\end{proof}

\begin{remark}
In the cases not covered by Propositions \ref{6gon ub toric2}, \ref{6gon ub 3bu} the polygon space
$\M(r)$ equipped with the bending action along the system of diagonals as in Figure \ref{fig:6gonsystem}
is not obtained from a toric Fano manifold by blowing up at toric fixed points, and so 
Theorem \ref{theorem LuBlowupUpperBounds} cannot be applied. 
Note however that all $\M(r)$ are obtained by a sequence of symplectic cuts\footnote{For the definition of symplectic cuts see \cite{Lerman}}
from the manifold associated to
the cuboid $C$, which is  $(\CP^1)^3$, with some scaling of Fubini-Study symplectic forms on each $\C \P^1$ factor.
It seems very natural to expect that the Gromov width of a compact symplectic manifold would not increase under the 
symplectic cut operation.
This would imply that the Gromov width of $\M(r)$ would be bounded above 
by the Gromov width of the manifold corresponding to $C$, which is $4 \pi r_1$. Together with 
Proposition \ref{proposition 6lowerbounds} this would prove that the Gromov width of $\M(r)$ is exactly $4 \pi r_1$.
\end{remark}

We now use a different argument to obtain the upper bound for the Gromov width of $6$-gons,
under different restrictions on the lengths $r_i$.
When $\Delta$ contains a whole facet $F$ of the cuboid $C$,
where one of the side lengths of $F$ is $2r_1$, then we obtain that the Gromov width of
the associated polygon space $\M(r)$ is at most $2 \pi \,2r_1$. We do this by showing the non-vanishing
of some Gromov--Witten invariant, as explained below. We are grateful to Dusa McDuff
for suggesting us this approach.

Suppose that the moment map image $\Delta$ for the toric (so K\"ahler) manifold $\M(r)$ contains a whole facet of the cuboid $C$,
where one of the side lengths is $2r_1$. 
Call this facet $F$, and let
$D_F := (d_1,d_2,d_3)^{-1} (F) \subset \M(r)$. This is a complex and symplectic submanifold of $\M(r)$.
 Note that as $r$ is generic, some neighborhood of $F$ in $C$ is also in $\Delta$. 
 Therefore some neighborhood of $D_F$ in $\M(r)$ is diffeomorphic
 to a neighborhood of $\C \P^1 \times \C \P^1 \times \{\pt \}$ 
in the K\"ahler manifold $(\C \P^1 \times \C \P^1 \times \C \P^1, 4r_1 \omega_{FS}\oplus 4r_3 \omega_{FS}\oplus 4r_5 \omega_{FS})$ 
corresponding to the cuboid $C$, by a diffeomorphism preserving the K\"ahler structure.
This means that the tangent space $T\M(r)$ in a neighborhood of $D_F$ splits as a sum of line bundles 
$T\C \P^1 \oplus T\C \P^1 \oplus T\C \P^1$, denoted later by $\mathbb{L}_1,\mathbb{L}_2$ and $\mathbb{L}_3$, and thus
%and that in a neighborhood of 
we can choose a compatible almost complex structure $J$ on $\M(r)$ which near
 $D_F$ is a direct sum 
$J=J_1 \oplus J_2 \oplus J_3$, where each $J_l$ is a complex structure on the respective copy of $\C \P^1$.

Let $A\in H_2(\M(r);\Z)$ be the homology class corresponding to the preimage (under the moment map $(d_1,d_2,d_3)$) 
of the edge of length $2r_1$. 
For simplicity of notation assume that this edge corresponds 
to the first $\C \P^1$ in $D_F=\C \P^1 \times \C \P^1 \times \{\pt \}$.
Note that the splitting of $T\M(r)$ near $D_F$ implies that the first Chern classes of line bundles $\mathbb{L}_1,\mathbb{L}_2$ and $\mathbb{L}_3$ evaluated on 
$[A]$ give 
$$c_1(\mathbb{L}_1)[A]=c_1(T \C\P^1)[ \C\P^1]=2, \qquad c_1(\mathbb{L}_2)[A]=0=c_1(\mathbb{L}_3)[A],$$ and thus $c_1(T\M(r))[A]=2$.
As these Chern numbers are not smaller than $-1$, by Lemma 3.5.1 from \cite{MSholom}, $J$ is regular for $A$.
Moreover, the equation \eqref{GWittendim} holds for $d=0$, and $k=1$,
 and 
one can consider the Gromov-Witten invariant, $\Phi_{A,1}([\pt])$, associated to the homology class $A$ and evaluated on the 
Poincar\'e dual to the homology class of a point. %, is an element of $\Z$. 
Observe that, $$\Phi_{A,1}([\pt]) =1 \neq 0.$$
Indeed, since $J|_{D_F}$ is a product, each $J$-holomorphic curve in $D_F$ must project to a $J$-holomorphic curve in each factor, 
and hence it must be of the form $\C \P^1\times \{ \pt \}$.
Therefore there is one such curve through every point $x\in D_F$. 
In general there 
might be other $J$-holomorphic curves in the manifold $\M(r)$ that go through the designated point but 
do not lie in $D_F$, which could count positively or negatively.
Note however that $D_F$ is $J$-holomorphic, and $A\cdot [D_F] = 0$.
Therefore positivity of intersections of $J$-holomorphic submanifolds (see \cite[Theorem 2.6.3]{MSholom} which, 
though stated in dimension $4$, also holds for higher dimensions) tells us that every $J$-holomorphic curve must lie entirely in $D_F$. 
Hence there 
are no other $J$-holomorphic curves and $\Phi_{A,1}([\pt]) =1$.
The non-vanishing of the above Gromov-Witten invariant (for the chosen regular $J$) implies that for a generic 
choice of an almost complex structure $J'$, the evaluation map is onto, and thus the Gromov width of $\M(r)$ is at most $\omega(A)=4 \pi r_1$ (Theorem \ref{thm ub jholom}).

Using this argument, we show that when the bending action on $\M(r)$ is toric and $\Delta$ contains one facet $F$ of the cuboid $C$ as above, with one edge of length $2r_1$
then the Gromov width of $\M(r)$ is at most   $2 \pi 2r_1$.
\begin{proposition}
 Let $r \in \R_+^6$ be generic, ordered non-decreasingly. If $\{1,2,6\}$ and $\{1,2,3,4\}$ are short, then the Gromov width of $\M(r)$ 
is at most $4 \pi r_1$.
\end{proposition}

\begin{proof}
We show that if $\{1,2,6\}$ and $\{1,2,3,4\}$ are short, then the top facet of the cuboid $C$ is in $\Delta$.
Assume first that $r_3 \neq r_4$. 
Consider the following reshuffling of $r$
$$
 \sigma(r)=(r_1, r_6, r_2, r_5, r_3,r_4). 
 $$
The bending action on $\M(\sigma(r))$ associated to the choice of diagonals in Figure \ref{fig:6gonsystem} is toric as $r_3 \neq r_4$ (so also $r_1 \neq r_6$ and $r_2 \neq r_5$).
The moment map image
$\Delta$ contains the ``top'' facet of the cuboid $C$ as it contains the vertices $v_5,v_6,v_7,v_8$ (see the table on page \pageref{chart} applied to $\sigma(r)$).
Then the image $\Delta$ of $\M(\sigma(r))$ contains the vertices $v_5,v_6,v_7,v_8$ if and only if
$\{1,2,3\}$, $\{1,3,4\}$ and $\{1,3,5,6\}$ are short for $\sigma(r)$, or, equivalently, if and only if 
$\{1,2,6\}$, $\{1,2,5\}$ and $\{1,2,3,4\}$ are short for $r$.
(Note that $\{1,2,6\}$ short implies  that $\{1,2,5\}$ is also short.)
Since the top facet of $\Delta$ has shortest edge of length $2 r_1$, by the argument above, if
$\{1,2,6\}$ and $\{1,2,3,4\}$ are short then the Gromov width of $\M(r)$ is less or equal 
to $4 \pi r_1$.

If $r_3=r_4$ then the bending action on $\M(\sigma(r))$ is not toric and the above argument does not apply.
In this case we proceed as before: consider the family $\M_t:=\M(r_1,r_2,r_3,r_4+t,r_5+t,r_6+t)$ and use the continuity argument, ``Moser's trick".

\end{proof}

Note that for $r \in \R^6_+$ ordered non-decreasingly, and such that
$$
\begin{array}{ll}
  \{1,2,6\} &\text{long} \\
\{4,6\} &\text{short} 
 \end{array}
\quad \quad \text{or} \quad \quad
 \begin{array}{ll}
  \{1,2,6\} &\text{short} \\
\{1,2,3,4\} &\text{long} \\
\{2,3,6\} &\text{long} 
 \end{array}
$$
none of our results for the upper bound applies. Hence in these cases we only have the 
lower bound as in Proposition \ref{proposition 6lowerbounds}.


\begin{thebibliography}{99}
\setlength{\labelwidth}{1.2cm} 
\bibitem[C16]{Caviedes} A. Caviedes Castro, \emph{Upper bound for the Gromov width of coadjoint orbits of compact Lie groups}, J. Lie Theory {\bf 26} (2016), 821--860, Copyright Heldermann Verlag 2016.

\bibitem[CLS11]{CLS} D. Cox, J. Little, H. Schenck, \emph{Toric Varieties}, Graduate Studies in Mathematics, {\bf 124}. 
American Mathemtical Society, Providence, RI, 2011. xxiv+841pp. ISBN: 978-0-8218-4819-7.

\bibitem[FHS]{FHS} M. Farber, J.-Cl. Hausmann, D. Schuetz, \emph{On the conjecture of Kevin Walker}, J. Topol. Anal., {\bf 1} (2009), 65 --86.

\bibitem[GM82]{gm} I.M.~Gelfand, R.D.~MacPherson, \emph{Geometry in Grassmannians
and a generalization of the dilogarithm}, Adv. in Math. {\bf 44} (1982), 279--312.

\bibitem[HK97]{HausmannKnutson} J.-C. Hausmann, A. Knutson, \emph{Polygon
spaces and Grassmannians}, Enseign. Math. (2) {\bf 43} (1997), 281--321

\bibitem[HK00]{HKequilateral} J.-C. Hausmann, A. Knutson, \emph{A limit of toric symplectic forms that has no 
periodic Hamiltonians}, Geom. Funct. Anal. {\bf 10} (2000), 556--562. 

\bibitem[KM96]{km} M.~Kapovich, J.J.~ Millson, \emph{The symplectic geometry of
polygons in Euclidean space,}  J. Differential Geom.  {\bf 44}  (1996),  no. 3,
479--513.

\bibitem[NU14]{nu}Y. Nohara, K. Ueda, \emph{Toric degenerations of integrable systems on Grassmannians and polygon spaces} 
Nagoya Math. J. {\bf 214} (2014), 125--168.

\bibitem[K06]{K} A. Kasprzyk, \emph{Toric Fano Varieties and Convex Polytopes}. 
Thesis (Doctor of Philosophy (PhD)). University of Bath. 2006

\bibitem[KT05]{KarshonTolman} Y. Karshon, S. Tolman, \emph{The Gromov width of complex 
Grassmannians}, Algebr. and Geom. Topol. {\bf 5} (2005), 911--922.

\bibitem[Kl94]{klyachko} A.~Klyachko, \emph{Spatial polygons and stable
configurations of points in the projective line}, Algebraic geometry and its
applications (Yaroslavl', 1992),  67--84, Aspects Math., E25, Vieweg,
Braunschweig, 1994.

\bibitem[L95]{Lerman} E. Lerman, \emph{Symplectic Cuts}, Math. Res. Lett. {\bf 2} (1995), 247--258. 

\bibitem[LMS13]{LMS} J. Latschev, D. McDuff, F. Schlenk, \emph{The Gromov width of 
$4$-dimensional tori}, Geometry and Topology {\bf 17} (2013) 2813--2853.

\bibitem[Lu06]{Lu} G. Lu, \emph{Symplectic capacities of toric manifolds and
related results}, Nagoya Math. J. {\bf 181} (2006), 149--184.

\bibitem[M14]{Mandini} A. Mandini, \emph{The
Duistermaat-Heckman formula and the cohomology of moduli spaces of polygons}, J.
Symplectic Geom. {\bf 12} (2014), 171--213.

\bibitem[McD98]{McDuff} D. McDuff, \emph{Lectures on Symplectic Topology},
IAS/Park City Math Series vol.7, ed. Eliashberg and Traynor, Amer. Math. Soc (1998).

\bibitem[McD09]{McDuffellipsoids} D. McDuff, \emph{Symplectic embeddings of 4-dimensional ellipsoids}, J. Topol. {\bf 2} (2009), 1--22.

\bibitem[MS04]{MSholom} D. McDuff, D. Salamon, \emph{J-holomorphic curves and symplectic topology}, Second edition. American Mathematical Society Colloquium Publications, 52. American Mathematical Society, Providence, RI, 2004. ISBN 0-8218-3485-1.

\bibitem[P14]{Pabiniak} M. Pabiniak, \emph{Gromov width of non-regular
coadjoint orbits of U(n), SO(2n) and SO(2n+1)},  Math. Res. Lett. {\bf 21} (2014), 187 -- 205.

\bibitem[Sch]{Schlenk} F. Schlenk, \emph{Embedding problems in symplectic geometry}, 
de Gruyter Expositions in Mathematics, vol. 40, Berlin, (2005), 
x+250 pp., ISBN 3-11-017876-1.

\bibitem[Sch05]{Sch} F. Schlenk, \emph{Packing symplectic manifolds by hand}, 
J. Symplectic Geom. {\bf 3} (2005), 313--340.

\bibitem[T95]{Traynor} L. Traynor, \emph{Symplectic Packing Constructions}, J.
Differential Geom. {\bf 42}, (1995), 411 -- 429.

\bibitem[Z10]{Zoghi} M. Zoghi \emph{The Gromov width of Coadjoint Orbits of Compact Lie Groups}, Ph.D. Thesis, University of Toronto, 2010.


\end{thebibliography}
 \end{document}